\documentclass{amsart}
\usepackage{amssymb,
enumitem,
mathrsfs,
hyperref,
tikz,
upgreek
}
\hypersetup{colorlinks=true}
\numberwithin{equation}{section}
\usepackage[T1]{fontenc}
\usepackage[shadow]{todonotes}

\renewcommand{\L}{\mathcal{L}}
\newcommand{\RR}{\mathbb{R}}
\newcommand{\NN}{\mathbb{N}}
\newcommand{\QQ}{\mathbb{Q}}
\newcommand{\pre}[2]{#2^{#1}}

\newcommand{\On}{\mathrm{On}}
\newcommand{\Nat}{\mathbb{N}}
\newcommand{\Sep}{\mathsf{Sep}}
\newcommand{\alzero}{{\aleph_0}}
\newcommand{\Nbhd}{\boldsymbol{N}}

\newcommand{\id}{\operatorname{id}}
\newcommand{\pow}{\mathscr{P}}
\newcommand{\Mod}{\operatorname{Mod}}
\newcommand{\lh}{\operatorname{lh}}

\newcommand{\dom}{\operatorname{dom}}

\newcommand{\cl}{{\operatorname{Cl}}}

\newcommand{\MS}{\mathfrak{M}}
\newcommand{\US}{\mathfrak{U}}

\newcommand{\T}[2]{\mathbb{T}_{#1}^{#2}}

\newcommand{\isom}{\cong}

\newenvironment{enumerate-(a)}{\begin{enumerate}[label={\upshape (\alph*)}, leftmargin=2pc]}{\end{enumerate}}
\newenvironment{enumerate-(a)-r}{\begin{enumerate}[label={\upshape (\alph*)}, leftmargin=2pc,resume]}{\end{enumerate}}
\newenvironment{enumerate-(a)-5}{\begin{enumerate}[label={\upshape (\alph*)}, leftmargin=2pc,start=5]}{\end{enumerate}}
\newenvironment{enumerate-(A)}{\begin{enumerate}[label={\upshape (\Alph*)}, leftmargin=2pc]}{\end{enumerate}}
\newenvironment{enumerate-(A)-r}{\begin{enumerate}[label={\upshape (\Alph*)}, leftmargin=2pc,resume]}{\end{enumerate}}
\newenvironment{enumerate-(i)}{\begin{enumerate}[label={\upshape (\roman*)}, leftmargin=2pc]}{\end{enumerate}}
\newenvironment{enumerate-(i)-r}{\begin{enumerate}[label={\upshape (\roman*)}, leftmargin=2pc,resume]}{\end{enumerate}}
\newenvironment{enumerate-(I)}{\begin{enumerate}[label={\upshape (\Roman*)}, leftmargin=2pc]}{\end{enumerate}}
\newenvironment{enumerate-(I)-r}{\begin{enumerate}[label={\upshape (\Roman*)}, leftmargin=2pc,resume]}{\end{enumerate}}
\newenvironment{enumerate-(1)}{\begin{enumerate}[label={\upshape (\arabic*)}, leftmargin=2pc]}{\end{enumerate}}
\newenvironment{enumerate-(1)-r}{\begin{enumerate}[label={\upshape (\arabic*)}, leftmargin=2pc,resume]}{\end{enumerate}}
\newenvironment{itemizenew}{\begin{itemize}[leftmargin=2pc]}{\end{itemize}}

\newtheorem{theorem}{Theorem}[section]
\newtheorem{lemma}[theorem]{Lemma}
\newtheorem{corollary}[theorem]{Corollary}
\newtheorem{proposition}[theorem]{Proposition}

\newtheorem{mainproblem}{Main Open Problem}
\newtheorem{fact}[theorem]{Fact}
\newtheorem{claim}{Claim}[theorem]
\theoremstyle{definition}
\newtheorem{defin}[theorem]{Definition}
\newtheorem{example}[theorem]{Example}
\newtheorem{notation}[theorem]{Notation}
\newtheorem{coding}{Coding}
\theoremstyle{remark}
\newtheorem{remark}[theorem]{Remark}

\begin{document}

\title[Classifying complete metric spaces]{Can we classify complete metric spaces \\ up to isometry?}
\date{\today}
\author[L.~Motto Ros]{Luca Motto Ros}
\address{Dipartimento di matematica \guillemotleft{Giuseppe Peano}\guillemotright, Universit\`a di Torino, Via Carlo Alberto 10, 10121 Torino --- Italy}
\email{luca.mottoros@unito.it}
 \subjclass[2010]{Primary: 03E15; Secondary: 54E50}
 \keywords{Polish (metric) spaces; isometry; classification; Borel reducibility}
\thanks{The author was supported by the Young
Researchers Program ``Rita Levi Montalcini'' 2012 through the project ``New
advances in Descriptive Set Theory'' for this research.}

\begin{abstract} 
We survey some old and new results concerning the classification of complete metric spaces up to isometry, a theme initiated by Gromov, Vershik and others. 
All theorems concerning separable spaces appeared in various papers in the last twenty years: here we tried to present them
in a unitary and organic way, sometimes with new and/or simplified proofs. The results concerning non-separable spaces (and, to some extent, the setup and techniques used 
to handle them) are instead new, and suggest new lines of investigation in this area of research. 
\end{abstract}

\maketitle

\tableofcontents

The purpose of this paper is threefold:
\begin{enumerate-(A)}
\item
Give a concise and self-contained presentation of the theory of Borel reducibility for equivalence relations, including some motivations for its development, accessible to non-experts 
in the field (Section~\ref{sec:Borelreducibility}).
\item
Illustrate how Borel reducibility (and descriptive set theory in general) may be used to gain new insight into some natural and interesting classification problems. 
We concentrate on the prominent example of classifying separable complete (i.e.\ Polish) metric spaces 
up to isometry, and survey the most important results in this area obtained by Gromov, Vershik, Kechris, Gao and many others, 
including some very recent progress on the long-standing open problem of determining the complexity of isometry between locally compact Polish metric spaces. When possible, we provide (sometimes simplified and self-contained) proofs of the most relevant results, or at least we overview the main ideas and methods involved in them (Sections~\ref{sec:survey} and~\ref{sec:treesisometry}).
\item
Show how a similar approach can be used to study the analogous classification problems for non-separable spaces (Section~\ref{sec:nonseparable}). These results are new and falls in the 
recent and fast-growing area of \emph{generalized descriptive set theory} --- see e.g.~\cite{Friedman:2011nx,AM} for more on this. This is the only part of the paper in which some knowledge of the most common set-theoretical 
axioms beyond the usual Zermelo-Fraenkel set theory \( \mathsf{ZFC} \) 
is helpful to fully appreciate the scope of validity of the results: nevertheless, the methods used to prove them are a generalization of techniques from~\cite{Camerlo:2015ar} 
which do not require any special background in set theory, and should thus be accessible to non-logicians as well.
\end{enumerate-(A)}
We conclude the paper with a brief discussion on how the (classical) theory of Borel reducibility could be further extended in order to be able to handle more and more classification problems naturally arising in other areas 
of mathematics (Section~\ref{sec:finaldiscussion}).

We tried to keep the material of the paper as self-contained as possible, sticking to standard notation and terminology, and reducing to a minimum the background in set theory required. 
For the reader's convenience, when possible we gave references to both 
the original papers where the results first appeared, and to the standard monographs~\cite{Kechris1995,gaobook} where some of them were presented in a more systematic way.

\section{Classification problems and Borel reducibility} \label{sec:Borelreducibility}

Consider the problem of classifying complete metric spaces, possibly belonging to some relevant subclass, up to isometry. This means that given a class \( \MS \) consisting of complete metric spaces, we want to assign 
\emph{complete invariants} to its elements, that is we seek a set \( I \) and a map \( \varphi \colon  \MS \to I \) such that for all \( X,Y \in \MS \)
\[ 
X \cong^i Y \iff \varphi(X) = \varphi(Y),
 \] 
where \( \cong^i \) denotes the equivalence relation of isometry between metric spaces. Clearly, such a classification is meaningful only if both the set of invariants \( I \) and the assignment map \( \varphi \) are reasonably simple and ``concrete''.%
\footnote{Of course simpleness and concreteness are vague concepts, but we will provide an exact mathematical formulation of what we have in mind later on.}

The first positive result of this kind is due to Gromov~\cite{gromov1999}.

\begin{theorem}[Gromov, see e.g.~{\cite[Theorem 14.2.1]{gaobook}}] \label{thm:gromov}
The problem of classifying compact metric spaces up to isometry is \emph{smooth} (that is, we can use reals as complete invariants).
\end{theorem}

Gromov's original proof of this result uses the so-called Gromov-Hausdorff distance; although Gromov's classification is not very difficult, the following (smooth) classification of compact metric spaces 
is somewhat easier to describe. 
Given a compact metric space \( X = (X,d) \), define for every \( n \in \Nat \) a map 
\[ 
\varphi_n \colon X^{n+1} \to \RR^{(n+1)^2} \qquad (x_0, \dotsc, x_n) \mapsto (d(x_i,x_j))_{0 \leq i,j, \leq n}.
 \] 
The maps \( \varphi_n \) are continuous, and since each \( X^{n+1} \) is compact then so is its image \( \varphi_n(X^{n+1}) \subseteq \RR^{(n+1)^2} \). Equip each hyperspace \( K(\RR^{(n+1)^2}) \) of compact subsets of \( \RR^{(n+1)^2} \) with the Vietoris topology (which is separable and completely metrizable, see e.g.~\cite[Section 14.F]{Kechris1995} for the definition and more details), so that also the countable product
\begin{equation} \label{eq:gromovinvariants}
I = \prod_{n \in \Nat} K(\RR^{(n+1)^2})
 \end{equation} 
is separable and completely metrizable. Then one can check that the map \(\varphi\) assigning to each compact metric 
space \( X \) the sequence
\begin{equation} \label{eq:gromovassignment}
\varphi(X) = (\varphi_n(X^{n+1}))_{n \in \Nat } \in I
 \end{equation}
is an assignment of complete invariants with respect to the isometry relation \( \cong^i \) on compact metric spaces (see e.g.~\cite[Theorem 14.2.1]{gaobook}).

One may argue that the above classification (which is, in a technical sense, equivalent to Gromov's one) is fully satisfactory because of the following three facts:
\begin{enumerate-(a)}
\item \label{gromov-a}
the space \( I \) of complete invariants from~\eqref{eq:gromovinvariants}, being separable and completely metrizable, is Borel isomorphic to the real line \( \RR \) (and thus could even be identified with it);
\item \label{gromov-b}
the assignment map \( \varphi \) from~\eqref{eq:gromovassignment} may be construed as a Borel map between the space \( K(\mathbb{U}) \) of all compact subspaces of the Urysohn space \( \mathbb{U} \) (which, up to isometry, contains all compact metric spaces), endowed with the Vietoris topology, and the space of invariants \( I \);
\item \label{gromov-c}
there is a Borel map%
\footnote{Such a Borel map may be obtained by first identifying through a Borel isomorphism the space 
\( I \) with the space \( \RR_{> 0 } \) of positive reals numbers, and then associating to each 
\( r \in \RR_{> 0} \) a compact metric space consisting of exactly two points at distance \( r \).}
 \( \psi \colon I \to K(\mathbb{U}) \) which assigns to distinct invariants non-isometric compact 
 metric spaces.
\end{enumerate-(a)}

Items~\ref{gromov-a}--\ref{gromov-b} show that the chosen assignment of complete invariants is natural and reasonably simple, while item~\ref{gromov-c}  essentially shows that no strictly simpler classification of compact metric spaces up to isometry may be obtained.

Theorem~\ref{thm:gromov} naturally raises the question whether a similar classification can be 
obtained for arbitrary \emph{Polish} (i.e.\ separable complete) \emph{metric spaces}. 
Unfortunately, Vershik pointed out in~\cite{vershik1998} 
that 
``the classification of Polish spaces up to isometry is an enormous task.
More precisely, this classification is not `smooth' in the modern terminology.''  Nevertheless, this is arguably not the end of the story for such a classification problem. For example, one may still find quite satisfactory a classification 
in which the complete invariants assigned are \emph{real numbers up to rational translation} (we will see in Theorem~\ref{thm:heineborelultra} that this is a necessary move if one wants to classify e.g.\ Heine-Borel Polish
ultrametric spaces); or one may 
feel that it would still be acceptable a classification obtained using \emph{countable structures} (graphs, linear orders, and so on) \emph{up to isomorphism} as complete invariants. And even when a solution to the given 
classification problem seems out of reach, one could still desire to have at least a way to measure how much complicated such a problem is when compared to other classification problems.

All these ideas are captured by the subsequent definition. In what follows, by a \emph{Borel space} \( X = (X, \mathcal{B}) \) we mean a set \( X \) equipped with a countably generated \(\sigma\)-algebra \( \mathcal{B} \) 
on \( X \)
which separates points (i.e\ for every \( x,y \in X \) there is \( B \in \mathcal{B} \) such that \( x \in B \) and \( y \notin B \)); the elements of \( \mathcal{B} \) are called \emph{Borel sets} of \( X \). 
By~\cite[Proposition 12.1]{Kechris1995}, this is 
equivalent to require that there is a separable metrizable topology \( \tau \) on \( X \) such that \( \mathcal{B} \) coincides with the \(\sigma\)-algebra of Borel sets generated by \( \tau \). When such a topology may be taken to 
be completely metrizable (i.e.\ a \emph{Polish} topology), then the Borel space \( (X,\mathcal{B}) \) is called \emph{standard}. A particularly important construction of a standard Borel space that will be used in this paper
 is given in the following example.
\begin{example} \label{xmp:EffrosBorel}
Given a topological space \( X \),  the collection \( F(X) \) of all its closed subsets can be equipped with the \(\sigma\)-algebra \( \mathcal{B}_{F(X)} \) generated by the sets of the form
\begin{equation} \label{eq:EffrosBorel}
\{ F \in F(X) \mid F \cap U \neq \emptyset \},
 \end{equation}
for \( U \subseteq X \) nonempty open; the \(\sigma\)-algebra \( \mathcal{B}_{F(X)} \) is called \emph{Effros Borel structure} of \( F(X) \). 
It turns out that if \( X \) is Polish, then \( (F(X), \mathcal{B}_{F(X)}) \) is a standard Borel space (see e.g.\ ~\cite[Theorem 12.6]{Kechris1995}).
\end{example}
 Clearly, every subspace
of a standard Borel space (with the induced subalgebra) is a Borel space and, conversely, every Borel space may be construed as a subspace of a suitable standard Borel space. 
It may be worth recalling that a subspace of a standard Borel space \( (X, \mathcal{B}_X ) \) is standard as well if and only if it belongs to \( \mathcal{B}_X \). Moreover, the class of standard Borel spaces is closed under 
(at most) countable products and disjoint unions.
Finally, a function \( f \) between the Borel spaces \( (X,\mathcal{B}_X) \) and \( (Y, \mathcal{B}_Y) \) is called \emph{Borel(-measurable)} if \( f^{-1}(B) \in \mathcal{B}_X \) for every \( B \in \mathcal{B}_Y \); when \( \mathcal{S} \) is a set of generators for the \( \sigma \)-algebra \( \mathcal{B}_Y \), this is clearly equivalent to requiring that \( f^{-1}(B) \in \mathcal{B}_X \) for every \( B \in \mathcal{S} \).

\begin{defin} \label{def:reducibility}
Let \( E \) and \( F \) be equivalence relations defined on Borel spaces \( X \) and \( Y \), respectively. We say that \( E \) is \emph{Borel reducible} to \( F \) (in symbols, \( E \leq_B F \)) if there is a Borel function \( f \colon X \to Y \) such that for every \( x,y \in X \)
\[ 
x \mathrel{E} y \iff f(x) \mathrel{F} f(y).
 \] 
A function \( f \) as above is called (Borel) \emph{reduction} of \( E \) to \( F \), and we denote this by \( f \colon E \leq_B F \). When \( E \leq_B F \) but \( F \nleq_B E \) we write \( E <_B F \), i.e.\ \( <_B \) is the strict part of \( \leq_B \). 
Finally, \( E \) and \( F \) are \emph{Borel bi-reducible} (in symbols, \( E \sim_B F \)) if both 
\( E \leq_B F \) and \( F \leq_B E \).
\end{defin}

\begin{remark}
In the current literature, Definition~\ref{def:reducibility} is usually given only for equivalence relations defined on \emph{standard} Borel spaces. This is probably because the theory of Borel reducibility becomes somewhat nicer under 
this simplification. However, if one wants to study classification problems naturally arising in other areas of mathematics our more general setup becomes necessary, as it is often the case that
certain natural classes of objects form Borel spaces which are not necessarily standard --- see e.g.\ Proposition~\ref{prop:discreteandlocallycompactarecoanalytic}. 
When we will need to specify the complexity of the domain of a given equivalence relation, 
we will speak of equivalence relations with standard Borel (analytic, coanalytic, ...) domain, with the obvious meaning. 
\end{remark}

In the classical context, the theory of Borel reducibility is primarily concerned with \emph{analytic} equivalence relations, i.e.\ equivalence relations \( E \) on a standard Borel space \( X \) such that \( E \) is an analytic 
subset of \( X \times X \). (Recall that a subset of a standard Borel space \( Z \) is \emph{analytic} if and only if it is a Borel image of some 
standard Borel space: clearly, all Borel subsets of \( Z \) are analytic, but if \( Z \) is uncountable then there are analytic subsets of \( Z \) which are not Borel.) We now adapt this to our more general setup.

\begin{defin} \label{def:analyticer}
An equivalence relation \( E \) on a Borel space \( X \) is \emph{analytic} if and only if \( E = A \cap (X \times X ) \) for some analytic \( A \subseteq Y \times Y \) with \( Y \supseteq X \) a standard Borel space.
\end{defin}

Clearly, all \emph{Borel} equivalence relations (i.e.\ equivalence relations \( E \) on a Borel space \( X \) which are Borel subsets of \( X \times X \)) are analytic. Furthermore, we will see 
in Proposition~\ref{prop:reductiontostandard} that nearly every analytic equivalence relation on a Borel space can be construed as the restriction to its domain of an analytic 
equivalence relation defined on a larger \emph{standard} Borel space.

Let us now come back to classification problems and briefly discuss how the theory of Borel reducibility between (analytic) equivalence relations is related to them.
Any classification problem is essentially an equivalence relation \( E \) on a space \( X \) (and viceversa, so that we can safely confuse the two notions). When \( X \) carries a natural Borel structure (in the sense described before Example~\ref{xmp:EffrosBorel}), the equivalence relation \( E \) naturally fits the setup of Definition~\ref{def:reducibility}. Under these circumstances, it is then natural to give the following definition 
 (which makes mathematically precise the statement of Theorem~\ref{thm:gromov}).

\begin{defin} \label{def:smooth}
An equivalence relation (or a classification problem) 
\( E \) on a Borel space \( X \) is called \emph{smooth} if there is a Borel space \( Y \) such that \( E \leq_B \id(Y) \), where \( \id(Y) \) is the identity relation on the space \( Y \).
\end{defin}

\begin{remark}
The usual definition of smoothness concerns only Borel equivalence relations \( E \) with \emph{standard} Borel domains, while in Definition~\ref{def:smooth} we are not imposing
any restriction on the Borel domain of the equivalence relation considered. However, 
our more general definition is coherent with the original one, in that any smooth equivalence relation (possibly with non-standard Borel domain) is the restriction of a smooth equivalence relation with standard Borel domain --- see Proposition~\ref{prop:reductiontostandard} and the ensuing remark.
Similar considerations apply to the subsequent definitions of essentially countable, classifiable by countable structures, and essentially orbit equivalence relations, as well as to the
observations concerning the \( \leq_B \)-maximality in such classes of certain elements with a standard Borel domain.
\end{remark}

Notice that since by~\cite[Theorem 15.6]{Kechris1995} any two uncountable standard Borel spaces are Borel isomorphic, in Definition~\ref{def:smooth} we may equivalently replace \( Y \) with any given uncountable standard Borel space, such as 
\( \RR \). Definition~\ref{def:smooth} thus identifies when the elements of \( X \) are ``concretely classifiable'' up to \( E \), i.e.\ when it is possible to assign in a Borel fashion complete invariants from a ``nice'' space (like \( \RR \)) 
to the elements of \( X \). 

As mentioned above, in some situations one may be content e.g.\ with a classification up to \( E \)-equivalence of the elements of \( X \) which uses reals up to rational translations as complete invariants: 
in light of Definition~\ref{def:reducibility}, this amounts to show that \( E \leq_B E_0 \), where \( E_0 \) is the Vitali equivalence relation on \( \RR \) defined by 
\[ 
r \mathrel{E_0} r' \iff r-r' \in \QQ .
\] 

More generally, one 
may accept to be able to assign complete invariants to the elements of \( X \) ``up to countably many mistakes''. 
In our setup, this would correspond to showing that \( E \leq_B F \) with \( F \) a \emph{countable Borel} equivalence relation, i.e.\ a Borel equivalence relation \( F \) on a (standard) Borel space such that each 
\( F \)-equivalence class has (at most) countably many elements. Notice that among countable 
Borel equivalence relations there is a \( \leq_B \)-maximal one, which is usually denoted by \( E_\infty \): in fact, \( E_\infty \) can be construed as the orbit equivalence 
relation induced by the shift action of the free group \( \mathbb{F}_2 \) with two generators on the space \( \pow(\mathbb{F}_2) \) of all subsets of \( \mathbb{F}_2 \) (see e.g.~\cite[Theorem 7.3.8]{gaobook}). Thus the concept of being able to assign 
complete invariants ``up to countably many mistakes'' is captured precisely by the following definition.

\begin{defin}
An equivalence relation \( E \) on a Borel space \( X \) is called \emph{essentially countable} if \( E \leq_B E_\infty \).
\end{defin}

Even more generously, one may still consider reasonable to classify the elements of \( X \) using countable structures (graphs, linear orders, and so on) up to isomorphism. Using Definition~\ref{def:reducibility}, we can again 
formulate this into the following definition.

\begin{defin} \label{def:classifiable}
An equivalence relation \( E \) on a Borel space \( X \) is called \emph{classifiable by countable structures} if \( E \leq_B {\cong_{\L}} \), where \( \cong_\L \) denotes the isomorphism relation between countable 
\( \L \)-structures for \( \L \) an arbitrary countable (relational) language.
\end{defin}

The isomorphism relation \( \cong_\L \) appearing in the above definition \label{pageref:models} may be construed as an equivalence relation on a Polish (hence standard Borel) space as follows. Let \( \L = \{ R_i \mid i < I \} \) with%
\footnote{We denote by \( \omega \) the first infinite ordinal. As customary in set theory, when useful we identify it with the set \( \NN \) of natural numbers.} 
\( I \leq \omega \) be the given language, and let \( n_i \) be the arity of the relational symbol \( R_i \). Up to isomorphism, each countable \( \L \)-structure has domain \( \Nat \), and thus can be identified (through 
characteristic functions of its predicates) with an element of 
\[ 
\Mod^{\alzero}_\L = \prod_{i < I} \pre{\left( \pre{n_i}{\Nat} \right)}{2}.
 \] 
Conversely,%
\footnote{For example, when \( L = \{ R \} \) is the graph language consisting of just one binary relational symbol, then \( \Mod^{\alzero}_\L = \pre{\Nat \times \Nat}{2} \), and each \( x \in \Mod^\alzero_\L \) gives rise to the \( \L \)-structure \( M_x = (\Nat, R^{M_x}) \), where for every \( n,m \in \Nat \) we set \( n \mathrel{R^{M_x}} m \iff x(n,m) =1 \). }
 each element \( x = (x_i)_{i < I} \in \Mod^\alzero_\L \) gives rise to a unique \( \L \)-structure \( M_x \) with domain \( \Nat \) such that for every \( i < I \) and \( \vec{a} \in \pre{n_i}{\Nat} \)
\[ 
M_x \models R_i[\vec{a}] \iff x_i(\vec{a}) = 1.
 \] 
Thus \( \Mod^\alzero_\L \) may be regarded as the space of all countable \( \L \)-structures (up to isomorphism). Moreover, since \( \Mod^\alzero_\L \) is naturally isomorphic to (a product of copies of) the Cantor space \( \pre{\Nat}{2} \), it can be equipped with a natural Polish topology and the corresponding standard Borel structure. Such topology, sometimes called \emph{logic topology}, is generated by the basic open sets determined by finite \( \L \)-structures with domain \( \subseteq \Nat \).

\begin{remark} \label{rmk:models}
In many cases, one is interested in a special subclass of \( \L \)-structures (e.g.\ graphs, linear orders, trees, ...): this in general amounts to considering a certain Borel subspace of \( \Mod^\alzero_\L \) closed under isomorphism. 
It is worth noticing that there is a close connection between the complexity of such a  Borel subspace and the possible axiomatization of the class of \( \L \)-structures under consideration. 
Indeed, by the Lopez-Escobar theorem 
(see e.g.~\cite[Theorem 16.8]{Kechris1995}) 
for every \( X \subseteq \Mod^\alzero_\L \) the following are equivalent:
\begin{itemizenew}
\item
\( X \) is Borel (i.e.\ a standard Borel space) and closed under isomorphism;
\item
\( X = \{ x \in \Mod^\alzero_\L \mid M_x \models \upvarphi \} \) for some \( \L_{\omega_1 \omega} \)-sentence \( \upvarphi \), where \( \L_{\omega_1 \omega} \) is the infinitary logic obtained from the usual first-order logic 
by further allowing the use of countable conjunctions and disjunctions.
\end{itemizenew}
It follows that most of the classes of countable structures we are going to consider, like e.g.\ graphs, linear orders, trees, and so on, form standard Borel subspaces of \( \Mod^\alzero_\L \).
\end{remark}

As for the case of countable Borel equivalence relations, also among the equivalence relations classifiable by countable structures there is a \( \leq_B \)-maximal one;
 by results of H.~Friedman and Stanley
(see e.g.~\cite[Sections 13.1--13.3]{gaobook}) such an equivalence relation may construed as any of the following:
\begin{itemizenew}
\item
the relation \( \cong_{\mathsf{GRAPHS}} \) of isomorphism between countable graphs;
\item
the relation \( \cong_{\mathsf{LO}} \) of isomorphism between countable linear orders;
\item
the relation \( \cong_{\mathsf{TREES}} \) of isomorphism between countable (descriptive set-theoretic) trees, where we call \emph{tree} any subset of the collection \( \pre{<\Nat}{\Nat} \) of finite sequences of natural 
numbers which is closed under initial segments and ordered by the subsequence (i.e.\ the inclusion) relation --- this corresponds to the special case of Definition~\ref{def:alpha-tree} where we set 
\( \alpha = \kappa = \omega \) (recall that we identify \( \omega \) with \( \NN \)).
\end{itemizenew}
Thus an equivalence relation \( E \) is classifiable by countable structures if and only if \( E \leq_B {\cong_{\mathsf{GRAPHS}}} \) (equivalently, \( E \leq_B {\cong_{\mathsf{LO}}} \), or 
\( E \leq_B {\cong_{\mathsf{TREES}}} \)). Notice that, contrarily to the case of essentially countable equivalence relations, an equivalence relation which is classifiable by countable structures needs not to be Borel.

The isomorphism relation \( \cong_\L \) coincides with the orbit equivalence relation induced by the so-called \emph{logic action} of the group \( S_\infty \) of permutations of \( \Nat \) on the space  \( \Mod^\alzero_\L \) of countable \( \L \)-structures --- see e.g.~\cite[Section 16.C]{Kechris1995} for more details. This suggests to further generalize Definition~\ref{def:classifiable} by considering arbitrary equivalence relations induced by continuous actions of Polish groups. In what follows, \( G \) is a \emph{Polish group} (i.e\ a topological group whose topology is separable and completely metrizable) acting in a continuous way on a Polish space \( Y \). The resulting \emph{orbit equivalence relation} will be denoted by \( E_G^Y \), where for all \( x,y \in Y \)
\[ 
x \mathrel{E_G^Y} y \iff \exists g \in G \, (g.x = y).
 \] 
(Here \((g,x) \mapsto g.x \) denotes the action of \( G \) on \( Y \).)

\begin{defin} \label{def:essorbit}
An equivalence relation \( E \) on a Borel space \( X \) is \emph{an essentially orbit equivalence relation} if there is a Polish group \( G \) acting continuously on a Polish space \( Y \) such that \( E \leq_B {E_G^Y} \).
\end{defin}

It can be shown that for every Polish group \( G \) there is a \( \leq_B \)-maximal element among the orbit equivalence relations of the form \( E_G^Y \), which is usually denoted by 
\( E_G^\infty \).%
\footnote{The relation \( E^\infty_G \) can be construed as the coordinatewise right action of \( G \) on the product \( \pre{\Nat}{F(G)} \) of countably many copies of the Effros Borel space \( F(G) \). 
Strictly speaking, this 
is just a Borel action of \( G \) on a standard Borel space: nevertheless, by~\cite[Theorem 4.4.6]{gaobook} it is possible to equip the latter with a Polish topology (generating the same Borel sets) which turns the 
given action of \( G \) into a continuous one.} 
Moreover, since Uspenski\v{\i} showed that the group \( H([0;1]^\NN) \) of homeomorphisms of the Hilber cube \( [0;1]^\NN \) is universal among Polish 
groups (i.e.\ every such group is isomorphic to a closed subgroup of \( H([0;1]^{\NN}) \), see~\cite[Theorem 9.18]{Kechris1995}), the equivalence relation \( E^\infty_{H([0;1]^{\NN})} \) is 
\( \leq_B \)-maximal among \emph{all} orbit equivalence relations, that is \( E_G^Y \leq_B E^\infty_{H([0;1]^{\NN})}  \) for every Polish group \( G \) acting continuously
on a Polish space \( Y \). As a consequence, \( E \) is an essentially orbit equivalence relation if and only if \( E \leq_B {E^\infty_{H([0;1]^{\NN})}} \). Of course, the latter condition provides an extremely loose classification for \( E \), as the equivalence relation \( {E^\infty_{H([0;1]^{\NN})}} \) is enormously complex; however, it still gives us some
nontrivial information on \( E \), e.g.\ it implies that all \( E \)-equivalence classes are Borel.

Since it can be shown that
\[ 
\id(\RR) <_B E_0 <_B E_\infty <_B  E^\infty_{H([0;1]^{\NN})},
 \] 
one sees that Definitions~\ref{def:smooth}--\ref{def:essorbit} determine larger and larger subclasses of 
the collection of all analytic equivalence relations: yet they do not exhaust it. For example, it is not hard to see that there are analytic 
equivalence relations \( E \) (with standard Borel domain) which are \( \leq_B \)-maximal among \emph{all} analytic equivalence relations: such relations, which are dubbed \emph{complete}, 
are thus the most complex analytic equivalence relations and 
cannot satisfy any of Definitions~\ref{def:smooth}--\ref{def:essorbit} because e.g.\ there must necessarily be \( E \)-equivalence classes which are proper analytic (hence non-Borel). 
Figure~\ref{fig:sumup} summarizes our discussion on the structure of analytic equivalence relations under 
\( \leq_B \) and can be used for future reference when we will determine the complexity of various concrete classification problems in the next sections.

\begin{center}
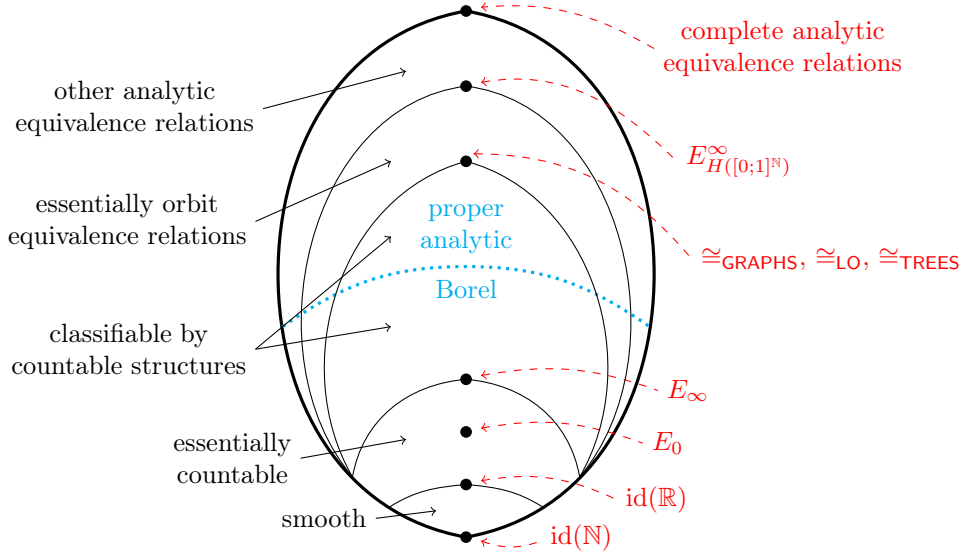
\begin{figure}
\begin{tikzpicture}[scale=1]
\draw [very thick] (0,0) to [out=10,in=270] (2.5,3.5) to [out=90,in=350] (0,7); 
\draw [very thick] (0,0) to [out=170,in=270] (-2.5,3.5) to [out=90, in=190] (0,7); 
\draw [fill] (0,0) circle [radius=0.07]; 
\draw [fill] (0,7) circle [radius=0.07]; 
\draw [red,dashed, ->] (2.5,6.5) to [out=150,in=5] (0.2,7.05);
\node [red, align=center,right] at (2.5,6.5) {complete analytic \\ equivalence relations};
\draw [red,dashed,<-] (0.2,-0.05) to [out=-20, in=-160] (1,-0.05);
\node [red,right] at (1,0) {\( \mathrm{id(\mathbb{N})} \)}; 
\draw [fill] (0,0.7) circle [radius=0.07]; 
\draw [red,dashed, <-] (0.2,0.75) to [out=10,in=155] (2,0.5);
\node [red,right] at (2,0.5) {\( \mathrm{id(\mathbb{R})} \)};
\draw (-1.02,0.4) to [out=32,in=185] (0,0.7) to [out=355,in=148] (1.02,0.4);
\draw [->] (-1.2,0.25) to (-0.4,0.35);
\node [left] at (-1.2,0.25) {smooth};
\draw [cyan, dotted, very thick] (-2.44,2.8) to [out=35,in=180]  (0,3.6) to [out=0,in=145] (2.44,2.8); 
\node [cyan] at (0,3.3) {Borel};
\node [cyan, align=center] at (0,4.1) {proper \\ analytic};
\draw (-1.51,0.8) to [out=80,in=185] (0,2.1) to [out=355,in=100] (1.51,0.8);
\draw [->] (-2.2,1.05) to (-0.8,1.3);
\node [align=center,left] at (-2.2,1.05) {essentially \\ countable}; 
\draw [fill] (0,2.1) circle [radius=0.07]; 
\draw [red,dashed, <-] (0.2,2.15) to [out=10,in=160] (2.55,1.95);
\node [red,right] at (2.55,1.95) {\( E_{\infty} \)};
\draw [fill] (0,1.4) circle [radius=0.07]; 
\draw [red,dashed, <-] (0.2,1.45) to [out=10,in=160] (2.35,1.25);
\node [red,right] at (2.35,1.25) {\( E_{0} \)};
\draw  (1.51,0.8) to [out=60,in=345]  (0,5); 
\draw  (-1.51,0.8) to [out=120,in=195] (0,5); 
\draw [fill] (0,5) circle [radius=0.07]; 
\draw [red,dashed, <-] (0.1,5.1) to [out=0,in=130] (3,3.7);
\node [red,right] at (3,3.7) {\( \cong_{\mathsf{GRAPHS}} \), \( \cong_{\mathsf{LO}} \), \( \cong_{\mathsf{TREES}} \)};
\draw [->] (-2.8,2.5) to (-1,2.8);
\draw [->] (-2.8,2.5) to (-1,4.0);
\node [align=center, left] at (-2.8,2.5) {classifiable by \\ countable structures};
\draw  (1.51,0.8) to [out=50,in=353] (0,6); 
\draw  (-1.51,0.8) to [out=130,in=187] (0,6); 
\draw [fill] (0,6) circle [radius=0.07]; 
\draw [red,dashed, <-] (0.2,6.05) to [out=10,in=135] (2.8,5);
\node [red,align=center,right] at (2.8,5) {\( E^\infty_{ H([0;1]^\mathbb{N})} \)};
\draw [->] (-2.8,4.2) to (-1,5);
\node [align=center, left] at (-2.8,4.2) {essentially orbit \\ equivalence relations};
\draw [->] (-2.7,5.7) to (-0.8,6.2);
\node [align=center, left] at (-2.7,5.7) {other analytic \\ equivalence relations};
\end{tikzpicture}
\caption{The structure of analytic equivalence relations under Borel reducibility, together with some of the most important subclasses and some equivalence relations which are used as milestones when determining the complexity of classification problems.} \label{fig:sumup}
\end{figure}
\end{center}

In all the above examples we interpreted the relation \( E \leq_B F \) for \( E \) and \( F \) equivalence relations on the Borel spaces \( X \) and \( Y \),
respectively, as a precise mathematical formulation of the following informal idea:
\begin{quote}
We can classify elements of \( X \) up to \( E \)-equivalence using as complete invariants the \( F \)-equivalence classes (or more precisely: using as complete invariants the elements of \( Y \) up to \( F \)-equivalence).
\end{quote}
As explained above, this allows us to obtain (weaker and weaker) solutions to the given classification problem associated to \( E \).
However, once a solution \(f \colon E \leq_B F \) has been found, one would also like to know whether 
such a classification is as simple as possible, i.e.\ whether there can be strictly simpler \( F' <_B F \) such that still \( E \leq_B F' \). One way to prove that a given classification \( E \leq_B F \) is (essentially) 
optimal is to show that indeed \( F \leq_B E \);
this clearly prevents the existence of an \( F' \) as above.%
\footnote{Results establishing that \( E \leq_B F \) for \( E \) a ``concrete'' classification problem and \( F \) a well-understood equivalence relation are often called \emph{classification theorems}.
In contrast, results establishing that \( F \leq_B E \), especially when \( F \) is quite complicated, are often referred to as \emph{anti-classification theorems}: this is because, as discussed, they show that no classification strictly 
simpler than \( F \) can be obtained.}
Thus the relation of bi-reducibility \( E \sim_B F \) can be interpreted as a mathematical formulation of the 
following assertion:
\begin{quote}
The elements of \( X \) can be classified up to \( E \)-equivalence using the \( F \)-equivalence classes as complete invariants, and no better classification is possible.
\end{quote}

In many cases, the inverse reducibility \( F \leq_B E \) of a given classification \( E \leq_B F \) must be proved by hand. However, there are special situations in which certain important and deep dichotomy theorems 
are helpful in this respect.
For example, suppose that we showed that a certain equivalence relation \( E \) with standard Borel domain is smooth, i.e.\ that \( E \leq_B \id(\RR) \). Since this implies that \( E \) is Borel (hence coanalytic), the following Silver's dichotomy allows us to prove that such a classification is optimal by simply checking that \( E \) has uncountably many equivalence classes.
\begin{theorem}[Silver, see e.g.~{\cite[Theorem 5.3.5]{gaobook}}]
Let \( E \) be a coanalytic equivalence relation with standard Borel domain. 
Then either there are countably many \( E \)-equivalence classes, or else \( \id(\RR) \leq_B E \).
\end{theorem}

In particular, this shows that there are essentially two types of smooth classification problems: those with (at most) countably many equivalence classes, and those which are Borel bi-reducible with \( \id(\RR) \).

Similarly, if we can show that \( E \leq_B E_0 \), then the following generalization of the Glimm-Effros dichotomy due to Harrington, Kechris and Louveau can be used to show the optimality of such a classification.
\begin{theorem}[Harrington-Kechris-Louveau, see e.g.~{\cite[Theorem 6.3.1]{gaobook}}] \label{thm:HKL}
Let \( E \) be a Borel equivalence relation with standard Borel domain. Then either \( E \leq_B \id(\RR) \), or else \( E_0 \leq_B E \).
\end{theorem}
Since \( E \leq_B E_0 \) already implies that \( E \) is Borel, to show that such a classification is optimal it is thus enough to show that \( E \) is not smooth and then apply Theorem~\ref{thm:HKL}.

In the same vein, if one can show that the given analytic equivalence relation \( E \) is complete (i.e.\ such that \( F \leq_B E \) for all analytic equivalence relations \( F \)), then \( E \) is arguably 
as complex as possible, and we can 
thus interpret this fact as saying that the elements of \( X \) are essentially unclassifiable up to 
\( E \)-equivalence --- no reasonable classification of the elements of \( X \) can be given, no matter how generous we are in allowing complicated invariants. For example, using this idea Louveau and Rosendal~\cite{louros} showed that many mathematical objects (including e.g.\ countable graphs) are unclassifiable up to the relevant notion of bi-embeddability (see also~\cite{cammarmot} for a strengthening of these results), while Ferenczi, Louveau and Rosendal~\cite{Ferenczi:2009fk} showed that separable Banach spaces are unclassifiable up to isomorphism, and that the same is true for (Abelian) Polish groups up to isomorphism. Many other natural classification problems have been tackled using these ideas,  e.g.\ the problem of classifying ergodic measure-preserving transformations of the unit interval with Lebesgue measure (see~\cite{For}), the problem of classifying separable \( C^* \)-algebras (completely solved in~\cite{Sab}), or the problem of classifying finitely-additive measures (a problem suggested by Maharam, see~\cite{BorDza} for some partial results in this direction), to mention a few.

The preorder \( \leq_B \) can also be used to compare the relative complexity of different classification problems. Indeed, suppose we are given two equivalence relations \( E \) and \( E' \) on corresponding Borel spaces. Then we can interpret the relation \( E \leq_B E' \) as:
\begin{quote}
The problem of classifying elements of \( X \) up to \( E \)-equivalence is not more complicated than the problem of classifying elements of \( X' \) up to \( E' \)-equivalence.
\end{quote}
In fact, if \( E \leq_B E' \) then any (Borel) assignment \( \varphi \colon X' \to I \) of complete invariants (possibly up to some equivalence relation \( F \) on \( I \)) to the elements of \( X' \) can be converted into a 
corresponding 
assignment of the same complete invariants to the elements of \( X \) by simply composing \( \varphi \) with any reduction \( f \colon E \leq_B E' \). 
Under this interpretation, the relation \( E \sim_B E' \) becomes a mathematical formulation of the following assertion:
\begin{quote}
The classification problems associated to \( E \) and \( E' \) have the same complexity.
\end{quote}

\section{The complexity of the isometry relation on separable spaces} \label{sec:survey}

We now come back to the problem of classifying complete metric spaces up to isometry, concentrating first on the separable case. 
In order to let this problem fit our general setup involving the Borel reducibility \( \leq_B \), 
we first need to construe the hyperspace of \emph{Polish} metric spaces as a (standard) Borel space: 
this is done through suitable codings, and there are at least two natural (but equivalent) ways to do it.

\begin{coding}[See e.g.~{\cite{Gao2003,Camerlo:2015ar}}] \label{coding1}
 We exploit the existence of a universal Urysohn space \( \mathbb{U} \), referring the reader to~\cite[Section 1.2]{gaobook} for the relevant definitions and 
proofs. Given any Polish metric space \( X \), using the Kat\v{e}tov construction one can canonically construct a Polish metric space \( \mathbb{U}_X \) 
such that for all Polish metric spaces \( X \) and \( Y \) 
\begin{enumerate-(1)}
\item
 \( \mathbb{U}_X \) contains (a canonical isometric copy of) \( X \), and every isometry \( \varphi \colon X \to Y  \) can be extended to an isometry \( \varphi^* \colon \mathbb{U}_X \to \mathbb{U}_Y \);
 \item
\( \mathbb{U}_X \) has the 
so-called \emph{Urysohn property}, whence \( \mathbb{U}_X \cong^i \mathbb{U}_Y \) for all Polish metric spaces \( X,Y \).
\end{enumerate-(1)} 
Let now \( \mathbb{U} \) be the space \( \mathbb{U}_\RR \): by the Urysohn property, a metric space is Polish if and only if it is isometric to a closed subspace of \( \mathbb{U} \). It is thus natural to regard the space \( F (\mathbb{U}) \)  of closed subspaces of \( \mathbb{U} \), equipped with its Effros Borel structure, as the standard Borel space of all Polish metric spaces. 
\end{coding}

\begin{coding}[See e.g.~{\cite{vershik1998,clemensisometry}}] \label{coding2} 
Any separable complete metric space \( X = (X,d) \) is uniquely determined (up to isomorphism) by any countable dense subset \( D \subseteq X \) together with the distances between the points in \( D \). More precisely, given any enumeration \( (x_n)_{n \in \Nat} \)  of \( D \), we can code \( X \) as the element \( c_X \) of \( \pre{\Nat \times \Nat}{\RR} \) defined by \( c_X(n,m) = d(x_n,x_m) \). Let
\begin{align*}
\mathcal{M} = \{ c \in \pre{\Nat \times \Nat}{\RR} \mid {} & \forall n,m \in \Nat \, (c(n,m) \geq 0 ) \text{ and } \forall n \in \Nat \, (c(n,n) = 0 ) \text{ and } {} \\ 
& \forall n,m \in \Nat \, (c(n,m) = c(m,n)) \text{ and } {}\\
& \forall n,m,k \in \Nat \, (c(n,k) \leq c(n,m)+c(m,k) \}
 \end{align*}
Clearly, \( \mathcal{M} \) is a closed subset of \( \pre{\Nat \times \Nat}{\RR} \), hence it is a Polish space. Moreover, the code \( c_X \) for the Polish metric space \( X \) described above belongs to \( \mathcal{M} \). Conversely, 
given \( c \in \mathcal{M} \) we 
can define \( M_c \) as the completion of the well-defined metric space \( (\Nat/E_c ,d_c) \), where \( \Nat/E_c \) denotes the quotient of \( \Nat \) by the equivalence relation \( E_c \) defined by \( n \mathrel{E_c} m \iff c(n,m) = 0 \),  and  \( d_c([n]_{E_c},[m]_{E_c}) = c(n,m) \), where \( [n]_{E_c} \) and \( [m]_{E_c} \) are the \( E_c \)-equivalence classes or \(n \) and \(m \), respectively. 
(Notice that by construction \( M_{c_X} \cong^i X \) for any Polish metric space \( X \).) Thus \( \mathcal{M} \) 
may be regarded as the Polish space of all%
\footnote{Technically speaking, \( \mathcal{M} \) contains only codes for \emph{nonempty} Polish metric spaces: this  can be easily fixed by e.g.\ adding some element 
of \( \pre{\Nat \times \Nat}{\RR} \setminus \mathcal{M} \) to \( \mathcal{M} \) and use it as a code for the empty space. This minor issue does not arise with Coding~\ref{coding1} because we already have \( \emptyset \in F(\mathbb{U}) \).}
Polish metric spaces.
\end{coding} 

The two hyperspaces \( F(\mathbb{U}) \) and \( \mathcal{M} \) of (codes for) Polish metric spaces described in Codings~\ref{coding1} and~\ref{coding2} are equivalent in 
the following technical sense: there are Borel functions
 $\Phi \colon \mathcal M \to F( \mathbb U )$ and $\Psi \colon F( \mathbb U )\to \mathcal M $ such that $\Phi (c)\cong^i M_c$ and $M_{\Psi (X)}\cong^i X$, for every $c \in \mathcal M $ and $X\in F( \mathbb U )$ 
 (thus, in particular, the maps \( \Phi \) and \( \Psi \) reduce the isometry relation to itself). 
 
 A map \( \Psi \) as above may easily be obtained 
as follows.
 
\begin{notation} \label{notation:psi_n}
Using the selection theorem for \( F(X) \) of Kuratowski-Ryll-Nardzewski
(see~\cite[Theorem 12.13]{Kechris1995}), we fix once and for all a sequence
of Borel functions \( \psi_n \colon F(\mathbb{U}) \to \mathbb{U} \) such that for
every \( \emptyset \neq X \in F(\mathbb{U}) \) the sequence \( ( \psi_n(X) )_{ n \in \Nat } \) is
an enumeration (which can be assumed without repetitions if $X$ is infinite)
of a dense subset of \( X \); the maps \( \psi_n \) will be called
\emph{Borel selectors}.
 Notice that if \( X \in F(\mathbb{U}) \) is discrete
then \(( \psi_n(X) )_{ n \in \Nat } \) is an enumeration of the whole \( X \).
\end{notation}

Given now \( X \in F(\mathbb{U}) \), let \( \Psi(X) \in \mathcal{M} \) be defined by
 \begin{equation} \label{eq:Psi}
\Psi(X)(n,m) = d(\psi_n(X),\psi_m(X)).
 \end{equation}
Such a \( \Psi \) is clearly as required. 

The construction of the map \( \Phi \colon \mathcal{M} \to F(\mathbb{U}) \) 
proceeds instead as follows. 
Given a Polish metric space, let \( \mathbb{U}_X \) be the space obtained from \( X \) through the Kat\v{e}tov construction as described in Coding~\ref{coding1}, and let
\begin{equation} \label{eq:phi_X} 
\varphi_X \colon \mathbb{U}_X \to \mathbb{U} = \mathbb{U}_\RR 
\end{equation}
be the isometry coming from the proof of the Urysohn theorem (see~\cite[Theorem 1.2.5]{gaobook}), so that, in particular,  \( \varphi_X(X) \cong^i X \) for every Polish metric space \( X \). 
It can then be shown (see~\cite[Section 14.1]{gaobook}) that the map \( \Phi \) defined by setting for every \( c \in \mathcal{M} \)
\begin{equation} \label{eq:Phi}
\Phi(c) = \varphi_{M_c}(M_c)
 \end{equation}
is as required.

\begin{remark}
It may be further checked that the map \( \Phi \) from~\eqref{eq:Phi} 
is injective (but not onto), and that with a minor modification of our construction the map \( \Psi \) from~\eqref{eq:Psi} may be turned into an injective one. Using these facts and the 
Borel Schr\"oder-Bernstein theorem~\cite[Theorem 15.7]{Kechris1995}, 
we can further improve the equivalence between the codings \( F(\mathbb{U}) \) and \( \mathcal{M} \) above by showing that indeed there is a \emph{Borel isomorphism} 
\( \Theta \colon \mathcal{M} \to F(\mathbb{U}) \) such that \( \Theta(c) \cong^i M_c \) for every \( c \in \mathcal{M} \) (see~\cite[Theorem 14.1.3]{gaobook}).
\end{remark}

Given the equivalence between the two spaces \( F(\mathbb{U}) \) and \( \mathcal{M} \) of codings  for Polish metric spaces presented above, 
\emph{from this point onward we will denote by \( \MS_\Sep \) the standard Borel space of all 
 Polish metric spaces} (where \( \Sep \) is for separable), without specifying which of the above codings we are actually using (although for technical reasons 
 in our proofs we will prefer to deal with \( F(\mathbb{U}) \) in most cases). It is straightforward to check that the isometry relation \( \cong^i \) is an analytic equivalence relation on \( \MS_\Sep \).

In what follows we will be interested in some natural subclasses of \( \MS_\Sep \), including the following:
\begin{itemizenew}
\item
\textbf{compact} Polish metric spaces: up to isometry, these are the spaces in \( K(\mathbb U) \) (see the discussion after Theorem~\ref{thm:gromov});
\item
\textbf{locally compact} Polish metric spaces (such as the finite-dimensional Euclidean spaces \( \RR^n \));
\item
\textbf{\( \sigma \)-compact} Polish metric spaces, i.e.\ spaces which are a union of countable many compact spaces: every locally compact Polish metric space is also \( \sigma \)-compact (but the converse does not hold);
\item
\textbf{countable} Polish metric spaces: all countable spaces are trivially \( \sigma \)-compact, but not necessarily locally compact;
\item
\textbf{discrete} Polish metric spaces: of course, every discrete space is also locally compact and countable (but the converse does not hold);
\item
Polish \textbf{ultrametric} spaces: recall that \( d \) is an \emph{ultrametric} if it satisfies the following strengthening of the triangular inequality:
\begin{equation} \label{eq:ultra} 
d(x,z) \leq \max \{ d(x,y), d(y,z) \}.
 \end{equation}
Examples of Polish ultrametric spaces are the Cantor space \( \pre{\Nat}{2} \) and the Baire space \( \pre{\Nat}{\Nat} \) (equipped with the usual distance 
\[
d(x,y) = \inf \{ 2^{-n}  \mid \forall i < n\, (x_i = y_i) \} 
\] 
for all \( x = (x_i)_{i \in \Nat} \) and \( y = (y_i)_{i \in \Nat} \)), and the space \( \QQ_p \) of \( p \)-adic numbers for any prime \( p \);
\item
\textbf{zero-dimensional} Polish metric spaces, i.e.\ spaces whose topology admits a basis consisting of clopen sets: clearly, every ultrametric Polish space is zero-dimensional, but there are zero-dimensional Polish metric 
spaces whose metric is not an ultrametric%
\footnote{However, for every zero-dimensional Polish metric space \( X = (X,d) \) there is a complete ultrametric \( d' \) on \( X \) such that \( d' \) is compatible with \( d \), i.e.\ it generates
the same topology.} 
(for example, take the Cantor ternary set \( E_{\frac{1}{3}} \subseteq [0;1] \) with the metric induced by \( \RR \));
\item
\textbf{Heine-Borel} Polish metric spaces, i.e.\ spaces whose closed bounded subsets are always compact (e.g.\ the real line \( \RR \)). 
\end{itemizenew}

Some of the above classes form standard Borel spaces, as shown in the next proposition.

\begin{proposition} \label{prop:standardBorelmetricspaces}
The following classes of Polish metric spaces are Borel subsets of the standard Borel space \( \MS_\Sep \), and thus they are standard Borel themselves:
\begin{enumerate-(1)}
\item \label{prop:standardBorelmetricspaces-1}
compact Polish metric spaces;
\item \label{prop:standardBorelmetricspaces-2}
Polish ultrametric spaces;
\item \label{prop:standardBorelmetricspaces-3}
Heine-Borel Polish metric spaces.
\end{enumerate-(1)}
\end{proposition}

\begin{proof} 
\ref{prop:standardBorelmetricspaces-1} The collection of compact Polish metric spaces is \( K(\mathbb{U}) \), which is easily seen to be a Borel subset of \( F(\mathbb{U}) \) (see~\cite[Exercise 12.11]{Kechris1995}).

\ref{prop:standardBorelmetricspaces-2} It is not hard to check that a metric \( d \) on \( X \) is an ultrametric if and only if the inequality~\eqref{eq:ultra} is satisfied by all points \( x,y,z \) in a given dense subset of \( X \). 
 Thus using the Borel selectors \( \psi_n \) we get that \( \emptyset \neq X  \in F(\mathbb{U}) \) is a Polish ultrametric space if and only if the following Borel condition holds (where \( d \) may indifferently be thought of as the distance on \( X \) or on the whole \( \mathbb{U} \) because 
 \( X \subseteq \mathbb{U} \)):
\[ 
\forall n,m,k \in \Nat \, (d(\psi_n(X),\psi_k(X)) \leq \max \{ d(\psi_n(X),\psi_m(X)), d(\psi_m(X),\psi_n(X)) \}).
 \] 

\ref{prop:standardBorelmetricspaces-3} Given \( x \in X \) and \( r \in \RR_{> 0} \), let \( B_r(x) = \{ y \in X \mid d(x,y) < r \} \) be the open ball of \( X \)  with center \( x \) and radius \( r \), and let \( \cl \) be 
the closure operator. It is not difficult to see that \( X \in F(\mathbb{U}) \) is Heine-Borel if and only if \( \rho(x) = + \infty \) \emph{for all} \( x \in X \), if and only if  \( \rho(x) = + \infty \) \emph{for some} 
\( x \in X \), where 
\[ 
\rho(x) = \sup \{ r \in \RR_{> 0} \mid \cl(B_r(x)) \text{ is compact} \}.
 \] 
We need the following claim.
\begin{claim} \label{claim:HB}
For every \( r \in \RR_{> 0 } \) and \( n \in \Nat \), the map
\[ 
F(\mathbb{U}) \to F(\mathbb{U}) \qquad X \mapsto \cl(B_r(\psi_n(X)))
 \] 
is Borel.
\end{claim}

\begin{proof}[Proof of the claim]
Let \( U \subseteq \mathbb{U} \) be nonempty and open: we must show that the set
\[ 
P_U = \{ X \in F(\mathbb{U}) \mid \cl(B_r(\psi_n(X))) \cap U \neq \emptyset \}
 \] 
is Borel. Since 
\begin{align*}
\cl(B_r(\psi_n(X))) \cap U \neq \emptyset & \iff B_r(\psi_n(X)) \cap U \neq \emptyset \\
& \iff D \cap B_r(\psi_n(X)) \cap U \neq \emptyset
\end{align*}
for \( D \) any dense subset of \( X \), we have that
\[ 
P_U = \{ X \in F(\mathbb{U}) \mid \exists k \in \Nat \, (d(\psi_k(X),\psi_n(X)) < r \text{ and } \psi_k(X) \in U) \},
 \] 
whence \( P_U \) is Borel.
\end{proof}
It follows from Claim~\ref{claim:HB} that the subsequent condition characterizing Heine-Borel Polish metric spaces is Borel:
\[
\forall n \in \Nat \, (\cl(B_n(\psi_0(X))) \in K(\mathbb{U})).  \qedhere
\]
\end{proof}

In contrast to Proposition~\ref{prop:standardBorelmetricspaces}, we are now going to show that some of the above mentioned natural classes of Polish metric spaces are not Borel subsets 
of \( \MS_\Sep \), and thus they form non-standard Borel spaces. 
We first fix some notation.
\begin{notation} \label{not:trees}
We denote by \( \mathrm{Tr} \) the Polish space of countable (descriptive set-theoretical) trees on \(\Nat\) (see~\cite[Exercise 4.32]{Kechris1995}), and by \( \mathrm{WF} \subseteq \mathrm{Tr}\) the set of \emph{well-founded trees}, where a tree \( T \subseteq \pre{<\Nat}{\Nat} \) is called well-founded if there is no \( x = (x_i)_{i \in \Nat} \in \pre{\Nat}{\Nat} \) such that \( x \restriction n =  (x_0, \dotsc, x_{n-1}) \in T \) for all \( 0 \neq n \in \Nat \). 
\end{notation}

Recall from~\cite[Theorem 27.1]{Kechris1995} that \( \mathrm{WF} \) is a proper coanalytic (hence non-Borel) subset of the Polish space \( \mathrm{Tr} \).

\begin{proposition}[Camerlo-Marcone-Motto Ros] \label{prop:discreteandlocallycompactarecoanalytic}
There is a Borel map \( f \colon \mathrm{Tr} \to \MS_\Sep \) such that for every \( T \in \mathrm{Tr} \):
\begin{enumerate-(1)}
\item
\( f(T) \) is a Polish ultrametric  space;
\item
if \( T \) is well-founded, then \( f(T) \) is discrete (hence countable, locally compact, and \(\sigma\)-compact);
\item
if \( T \) is not well-founded, then \( f(T) \) is not \(\sigma\)-compact (hence neither countable nor locally compact nor discrete).
\end{enumerate-(1)}
In particular, the classes of discrete, countable, locally compact, and \(\sigma\)-compact Polish metric spaces are all non-Borel subsets of \( \MS_\Sep \), and thus they are non-standard Borel spaces. The same applies to the corresponding subclasses of Polish ultrametric spaces.
\end{proposition}

\begin{proof}
Given \( T \in \mathrm{Tr} \), let \( X_T \) be the subspace of \( \pre{\Nat}{\Nat} \) consisting of those \( x = (x_i)_{i \in \Nat} \) such that:
\begin{itemizenew}
\item
there is \( N_x \leq \omega \) such that \( x_i \neq 0 \) for all \( i < N_x \) and \( x_i = 0 \) for all \( i \geq N_x \) (when \( N_x = \omega \), this just means that \( x_i \neq 0 \) for all \( i \in \Nat \));
\item
for each \( n < N_x \), the sequence \( (x_{2i}-1)_{2i \leq n} \) belongs to \( T \). 
\end{itemizenew}
It is not hard to check that \( X_T \) is always closed subspace of \( \pre{\Nat}{\Nat} \), hence it is a Polish ultrametric space. Moreover, each \( x \in X_T \) with \( N_x < \omega \)  is isolated in \( X_T \), and \( \{ x \in X_T \mid N_x < \omega \} \) is dense in \( X_T \). The latter observation  allows us to straightforwardly construe the map \( T \mapsto X_T \) as a Borel map \( f \colon \mathrm{Tr} \to \MS_\Sep \).

Assume first that \( T \) is well-founded. Then there is no \( x \in T \) with \( N_x = \omega \), as otherwise the sequence \( (x_{2i}-1)_{i \in \Nat} \) would contradict the fact that \( T \) is well-founded. Therefore \( N_x < \omega \) for all \( x \in X_T \), and hence \( X_T \) is discrete.

Assume now that \( T \) is not well-founded, and let \( (x_i)_{i \in \Nat } \) be a witness of this. Let
\[ 
F_x = \{ y = (y_i)_{i \in \Nat} \in X_T  \mid \forall j \in \Nat \, (y_{2j} = x_j+1) \}.
 \] 
(In particular, \( N_y = \omega \) for all \( y \in F_x \).)	
The set \( F_x \) is clearly closed in \( X_T \) and homeomorphic to \( \pre{\Nat}{\Nat} \) via the map \( (y_i)_{i \in \Nat} \mapsto (y_{2j+1} - 1)_{j \in \Nat} \): thus \( X_T \) is not \(\sigma\)-compact by a theorem of Hurewicz (see~\cite[Theorem 7.10]{Kechris1995}).
\end{proof}

\begin{remark} \label{rmk:nonBorel}
It is not hard to see using the methods of descriptive set theory (see~\cite{Kechris1995}) that the classes of discrete, countable and locally compact Polish metric spaces are all coanalytic, and thus their complexity is fully 
determined by Proposition~\ref{prop:discreteandlocallycompactarecoanalytic}. As for \( \sigma \)-compact spaces, Proposition~\ref{prop:discreteandlocallycompactarecoanalytic} shows that they form a so-called coanalytic-hard 
class, but we do not know its exact complexity: the best upper bound we are aware of is the one given by the definition of \(\sigma\)-compactness, which shows that it is at least a \( \mathbf{\Sigma}^1_2 \) class 
(i.e.\ the projection of a coanalytic set). Similarly, we do not know the complexity of the class of zero-dimensional Polish metric spaces: an upper bound is again \( \mathbf{\Sigma}^1_2 \), but it may be not  sharp.
\end{remark}

Let us now come back to the problem of classifying Polish metric spaces up to isometry. In their seminal paper~\cite{Gao2003}, Gao and Kechris showed that the isometry relation \( \cong^i \) (on the whole \( \MS_\Sep \))
is Borel reducible to an orbit equivalence relation.

\begin{theorem}[Gao-Kechris, see e.g.~{\cite[Theorem 14.2.3]{gaobook}}] \label{thm:allPolishupperbound}
The relation of isometry on Polish metric spaces is an essentially orbit equivalence relation.
\end{theorem}

\begin{proof}
Consider the natural action of \( \mathrm{Iso}(\mathbb{U}) \), the group of isometries of the Urysohn space \( \mathbb{U} \) equipped with the 
pointwise convergence topology, on the space \( F(\mathbb{U}) \), that is for \( g \in \mathrm{Iso}(\mathbb{U}) \) and \( X \in F(\mathbb{U}) \) let
\[ 
g.X = g(X).
 \] 
Notice that event though it is obvious that if \( X \mathrel{E_{\mathrm{Iso}(\mathbb{U})}^{F(\mathbb{U})}} Y \) then \( X \cong^i Y \), the converse may badly fail --- 
not all partial isometries \(\varphi\) between two infinite elements of \( F( \mathbb{U}) \) 
can be extended to an isometry of the whole \( \mathbb{U} \): thus \( \cong^i \) is strictly coarser than the orbit equivalence relation 
\( E_{\mathrm{Iso}(\mathbb{U})}^{F(\mathbb{U})}\). However, we claim that the Borel map 
\[ 
G =  \Phi \circ \Psi \colon F(\mathbb{U}) \to F(\mathbb{U}) 
\] 
(where \( \Psi \) and \( \Phi \) are 
as in~\eqref{eq:Psi} and~\eqref{eq:Phi}, respectively) is a reduction of \( \cong_i \) to \( E_{\mathrm{Iso}(\mathbb{U})}^{F(\mathbb{U})} \). 

Let \( X,Y \in F(\mathbb{U}) \). If 
\( G(X) \mathrel{E_{\mathrm{Iso}(\mathbb{U})}^{F(\mathbb{U})}} G(Y) \), then \( G(X) \cong^i G(Y) \), whence also \( X \cong^i Y \) because \( G(X) = (\Phi \circ \Psi)(X) \cong^i M_{\Psi(X)} \cong^i X \) and 
\( G(Y) = (\Phi \circ \Psi)(Y) \cong^i M_{\Psi(Y)} \cong^i Y \). Conversely, assume that 
\( X \cong^i Y \), so that there is an isometry \( \varphi \colon M_{\Psi(X)} \to M_{\Psi(Y)} \). By the Kat\v{e}tov construction, the map \( \varphi \) can be extended to an isometry 
\( \varphi^* \colon \mathbb{U}_{M_{\Psi(X)}} \to \mathbb{U}_{M_{\Psi(Y)}} \).
Let \( \varphi_{M_{\Psi(X)}} \) and \( \varphi_{M_{\Psi(Y)}} \) be as in~\eqref{eq:phi_X}: then \(  \varphi_{M_{\Psi(Y)}} \circ \varphi^* \circ \varphi_{M_{\Psi(X)}}^{-1} \in \mathrm{Iso}(\mathbb{U}) \) witnesses
\[ 
\varphi_{M_{\Psi(X)}}(M_{\Psi(X)}) \mathrel{E_{\mathrm{Iso}(\mathbb{U})}^{F(\mathbb{U})}} \varphi_{M_{\Psi(Y)}}(M_{\Psi(Y)}),
 \] 
whence \( G(X) \mathrel{E_{\mathrm{Iso}(\mathbb{U})}^{F(\mathbb{U})}} G(Y) \) because by definition
\( G(X) = (\Phi \circ \Psi)(X) = \varphi_{M_{\Psi(X)}}(M_{\Psi(X)}) \) and \( G(Y) = (\Phi \circ \Psi)(Y) = \varphi_{M_{\Psi(Y)}}(M_{\Psi(Y)}) \).
\end{proof}

Since, as already observed, orbit equivalence relations may be very complex, it is natural to ask whether the classification provided by Theorem~\ref{thm:allPolishupperbound} is optimal: this is indeed the case, as shown independently by Gao-Kechris~\cite{Gao2003} and Clemens~\cite{clemensisometry}.

\begin{theorem}[Clemens, Gao-Kechris, see e.g.~{\cite[Corollary 14.3.6]{gaobook}}] \label{thm:lowerboundallPolish}
For every Polish group \( G \) acting continuously on a Polish space \( Y \) we have
\[ 
E_G^Y \leq_B {\cong^i}.
 \] 
Thus the isometry relation \( \cong^i \) on the whole \( \MS_\Sep \) is Borel bi-reducible with \( E^\infty_{H([0;1]^{\NN})} \), the most complex orbit equivalence relation.
\end{theorem}

The proofs of Gao-Kechris and Clemens of Theorem~\ref{thm:lowerboundallPolish} are quite different. The former is quite complicated and consists in three main steps:
\begin{enumerate-(1)}
\item
show that the orbit equivalence relation \( E^\infty(\mathbb{U}) \) induced by the action of \( \mathrm{Iso}(\mathbb{U}) \) on \( \prod_{n \geq 1} F(\mathbb{U}^n) \) defined by
\[ 
\varphi . (R_n)_{n \geq 1} = (\varphi^{(n)}(R_n))_{n \geq 1}
 \] 
is \( \leq_B \)-maximal among orbit equivalence relations, where \( \varphi^{(n)} \colon \mathbb{U}^n \to \mathbb{U}^n \) is the Cartesian product 
\( \underbrace{\varphi \times \dotsc \times \varphi}_{n \text{ times}} \) (see~\cite[Section 2.E]{Gao2003});
\item
prove that for every Polish metric space \( X \)
\[ 
E^1(X) \leq_B {\cong}^i,
 \] 
where \( E^1(X) \) is the orbit equivalence relation induced by the action of \( \mathrm{Iso}(\mathbb{U}) \) on \( \pre{\Nat}{(F(X))} \) given by
\[ 
\varphi . (C_i)_{i \in \Nat} = (\varphi(C_i))_{i \in \Nat}
 \] 
(see~\cite[Section 2.G]{Gao2003});
\item
show that
\[ 
E^\infty(\mathbb{U}) \leq_B E^1(X)
 \] 
where \( X \) is the completion of \( \mathbb{U}^\infty = \bigcup_{n \in \Nat} \mathbb{U}^{(3^n)} \) (see~\cite[Section 2.H]{Gao2003}).
\end{enumerate-(1)}

Clemens' proof is more elementary and informative because, as we will see, it provides additional information on the restriction of \( \cong^i \) to certain subclasses of \( \MS_\Sep \). 
With this other approach, one aims at directly showing that \( E^Y_G \leq_B \cong_i \) for any orbit equivalence relation \( E_G^Y \) as in Theorem~\ref{thm:lowerboundallPolish}.
By~\cite[Theorem 2.6.6]{Becker:1996uq} there is always a continuous action of \( G \) on a compact Polish metric spaces \( Y' \) such that \( E^Y_G \leq_B E^{Y'}_G \), thus we can assume that \( Y \) itself be compact with metric \( d_Y \) bounded by \( 1 \). Fix a compatible left-invariant \( d_G \) on \( G \) such that:
\begin{enumerate-(a)}
\item
\( d_G \) is again bounded by \( 1 \) (which can always be obtained by replacing \( d_G \) with \( \frac{d_G}{1+d_G} \) if necessary);
\item \label{cond:metriconG}
for every \( y \in Y \) and \( g,h \in G \)
\[ 
d_G(g,h) \geq \frac{1}{2}d_Y (g^{-1} . y, h^{-1} . y).
 \] 
(To achieve this condition it is enough to check that since \( Y \) is compact, by~\cite[Lemma 10]{clemensisometry} we can replace if necessary 
\( d_G \) with the new metric \( d'_G \) defined by 
\[ 
d'_G(g, h) = \frac{1}{2} d_G(g, h) + \frac{1}{2} \sup \{d_Y(g^{-1} . y, h^{-1} . y) \mid y \in Y \} ;
\] 
notice that if \( d_G \) is complete, then so is \( d'_G \).)
\end{enumerate-(a)}

The key idea is then to define a Borel map
\[ 
Y \to \MS_\Sep \qquad z \mapsto X_z = (X_z,d_z)
 \] 
so that the \( G \)-orbit of \( z \) is somewhat encoded in the set of distances realized by \( d_z \). When e.g.\ \( d_G \) is complete and \( Y \subseteq \RR \), this can be done simply by letting \( X_z = G \cup \{ x^* \} \) 
(where \( x^* \) is a new point not belonging to \( G \)) and setting 
\begin{align*} 
d_z(g,h) & =  d_G(g,h) && \text{for every }g,h \in G \\
 d_z(x^*,g) & =  1 + \frac{1}{2} \, g^{-1} . z && \text{for every } g \in G. 
 \end{align*} 
 (The fact that \( d_z \) satisfies the triangular 
inequality follows from condition~\ref{cond:metriconG}.) On the one hand, if \( z_1,z_2 \in Y \) are not in the same \( G \)-orbit then the sets \(  \{ d_{z_i}(x^*,g) \mid g \in G \} \) of distances  realized by \( x^* \) (for 
\( i=1,2\)) are distinct, and since any isometry between \( X_{z_1} \) and \( X_{z_2} \) would send \( x^* \) to itself, this shows that \( X_{z_1} \not\cong^i X_{z_2} \). On the other hand, if \( z_2 = h . z_1 \) for some 
\( h \in G \), then the map \( f \colon X_{z_1} \to X_{z_2} \) sending \( x^* \) to itself and \( g \in G \) to \( hg \in G \) is an isometry (this uses the left-invariance of \( d_G \)).

The construction in the general case is a little bit more delicate.%
\footnote{The argument we are going to present here does not explicitly appear elsewhere, but it is suggested by Clemens himself  at the beginning of~\cite[Section 2]{clemensisometry}. The construction 
from~\cite{clemensisometry} (see also~\cite[Section 14.3]{gaobook}) is actually a further elaboration of the same argument: such more complicated construction can  easily be generalized in order 
to obtain more results which, unfortunately, we have to omit here for the sake of conciseness.}
 First of all, since \( d_G \) cannot be assumed to be complete,  in the definition of the spaces \( X_z \) we will have to replace \( G \) with its \( d_G \)-completion \( \widehat{G} \). 
The second obstacle is that possibly \( Y \) is not even continuously embeddable in \( \RR \), so it is not possible to literally encode a \( G \)-orbit using distances from a single distinguished point \( x^* \).
To overcome this difficulty, we fix  an enumeration 
\( (y_n)_{n \in \Nat} \) of a countable dense  \( D \subseteq Y \), and simultaneously encode all the distances from the given \( g^{-1}. z \) to each element of \( D \) 
(this suffices because two points of \( Y \) with identical distances from all \( y_n \)'s must in fact coincide); this move requires us to use countably many (rather than one) distinguished points for the encoding. 

\begin{proof}[Proof of Theorem~\ref{thm:lowerboundallPolish}]
Let \( Y \), \( d_Y \), \( d_G \), \( \widehat{G} \),  and \( (y_n)_{n \in \Nat} \) be as in the discussion above.
Fix \( z \in Y \). Consider the metric \( d_z \) on \( G \cup \{ x^*_n \mid n \in \Nat \} \) (where the \( x^*_n \) are all distinct and not in \( G \)) defined by the following clauses:
\begin{align*}
d_z(g,h) & =  d_G(g,h) && \text{for every }g,h \in G \\
d_z(x^*_n,x^*_m) & =  \max \{ n+2, m+2 \} && \text{for every }n,m \in \Nat\\
d_z(x^*_n,g) & =  (n+2)+\frac{1}{2}d_Y(y_n,g^{-1} . z) && \text{for every }n \in \Nat \text{ and } g \in G.
\end{align*}
(By condition~\ref{cond:metriconG} we get that \( d_z \) satisfies the triangular inequality, and thus it is a metric.) Let \( X_z \in \MS_\Sep \) be the completion of \( G \cup \{ x^*_n \mid n \in \Nat \} \) with respect to 
\( d_z \), and notice that \( X_z \) may be construed as a space on \( \widehat{G} \cup  \{ x^*_n \mid n \in \Nat \} \) 
because the points \( x^*_n \) are all isolated by \( d_z \). We claim that the Borel map 
\( z \mapsto X_z \) reduces \( E^Y_G \) to \( \cong^i \). Fix \( z_1,z_2 \in Y \). 

If \( z_2 = h . z_1 \) for some \( h \in G \), then the isometry of \( G \cup \{ x^*_n \mid n \in \Nat \} \) sending each \( x^*_n \) to itself and 
\( g \in G \) to \( hg \in G \) can clearly be extended to an isometry between \( X_{z_1} \) and \( X_{z_2} \). 

Conversely, let \( \varphi \colon M_{z_1} \to M_{z_2} \) be an isometry. Since points in \( \widehat{G} \) realize \( d_{z_i} \)-distances 
\( \leq 1 \) while the points \( x^*_n \) do not (for \( i = 1,2 \)), we have that \( \varphi \restriction \widehat{G} \) is an isometry between \( (\widehat{G}, d_{z_1}) \) and \( (\widehat{G}, d_{z_2}) \)
and that \( \varphi \) sends points of the form \( x^*_n \) to points of the same form. Moreover, since each \( x^*_n \) realizes only \( d_{z_i} \)-distances
\( \geq n+2 \) and at least a \( d_{z_i} \)-distance \( < n+3 \) (for \( i = 1,2 \)), we further have \( \varphi(x^*_n) = x^*_n \) for every \( n \in \Nat \). Notice also that there must be \( g \in G \subseteq \widehat{G} \) 
such that \( \varphi(g) \in G \): this is because \( \varphi \restriction \widehat{G} \) is in particular a homeomorphism and the subspace \( G \) of \( \widehat{G} \), being Polish, is \( G_\delta \) and dense in \( \widehat{G} \), hence
comeager in it. We claim that for such a \( g \) 
\[ 
g^{-1} . z_1 = \varphi(g)^{-1} . z_2 , 
\]
whence \( z_1 \mathrel{E^Y_G} z_2 \). Indeed, since \( \varphi \) is an isometry we have (for every \( n \in \Nat \))
\begin{multline*}
d_Y(y_n, g^{-1} . z_1) = 2 (d_{z_1}(x^*_n, g) - n -2) = 2 (d_{z_2}(\varphi (x^*_n), \varphi( g)) - n -2) \\ = 2 (d_{z_2}(x^*_n, \varphi( g)) - n -2) = d_Y(y_n, \varphi(g)^{-1} . z_2),
 \end{multline*}
so that \(  g^{-1} . z_1 \) and \( \varphi(g)^{-1} . z_2 \) have the same distances from all the points in the fixed dense set \( D \subseteq Y \). 
\end{proof}

Of course, the complexity of \( \cong^i \) may drastically drop down when restricting our attention to certain special subclasses of \( \MS_\Sep \), as illustrated by Theorem~\ref{thm:gromov}: other simple examples of this phenomenon are given by the following result from~\cite{Gao2003}.

\begin{theorem}[Gao-Kechris, see~{\cite[Proposition 4.2 and Theorem 4.4]{Gao2003}}] \label{thm:upperboundultrametric}
\begin{enumerate-(1)}
\item \label{thm:upperboundultrametric-1}
The isometry relation on all countable (hence, in particular, also on discrete) Polish metric spaces is classifiable by countable structures.
\item \label{thm:upperboundultrametric-2}
The isometry relation on all ultrametric Polish spaces is classifiable by countable structures.
\end{enumerate-(1)}
\end{theorem}

\begin{proof}
\ref{thm:upperboundultrametric-1}
Let \( \mathcal{L} = \{ P_q \mid q \in \QQ_{> 0} \} \) be a first-order language in which each \( P_q \) is a binary relation. 
Let \( X = (X,d) \) be a discrete Polish metric space, and recall that in this case  
\( (\psi_n(X))_{n \in \Nat} \) is an enumeration of the whole \( X \). 
Let \( G_X = (\Nat, (P_q^{X})_{q \in \QQ_{>0}}) \) be the countable \( \mathcal{L} \)-structure obtained
 by setting for every \(n,m\in \Nat \) and every \( q \in \QQ_{>0} \)
\[ 
P_q^{X}(n,m) \iff d(\psi_n(X),\psi_m(X)) < q.
 \] 
The map \( X \mapsto G_X \) is easily seen to be Borel, so it remains to check that it reduces isometry to isomorphism. 

Fix two discrete Polish metric spaces \( X = (X,d_X) \) and \( Y = (Y,d_Y) \).
Clearly, if \( X \cong^i  Y \) then \( G_X \cong G_Y \). Conversely, let \( f \colon G_X \to G_Y \) be an isomorphism. We first check that for every \( n,m \in \Nat \)
\begin{equation} \label{eq:distpres}
d_X(\psi_n(X), \psi_m(X)) = d_Y(\psi_{f(n)}(Y), \psi_{f(m)}(Y)).
 \end{equation}
Indeed, for every \( q \in \QQ_{> 0} \) and \( n,m \in \Nat \)
\begin{multline*}
d_X(\psi_n(X),\psi_m(X)) < q \iff P_q^{X}(n,m) \\ \iff P_q^{Y}(f(n),f(m)) \iff d_Y(\psi_{f(n)}(Y), \psi_{f(m)}(Y)) < q.
 \end{multline*}
Let now \( \varphi \colon X \to Y \) be defined by setting \( \varphi(\psi_n(X)) = \psi_{f(n)}(Y)\) for all \( n \in \Nat \):
using~\eqref{eq:distpres} one can easily check that \( \varphi \) is a well-defined bijection and that it is distance preserving, hence we are done.

\ref{thm:upperboundultrametric-2}
We use some special properties of ultrametric spaces. Since these features will also be used in the sequel, for the reader's convenience we sum up all of them in the following fact
(the proofs are left to the reader as an easy but instructive exercise).
\begin{fact} \label{fct:basicsonultra}
Let \( X = (X,d) \) be an ultrametric space and denote by \( B_q(x) \) the open ball of \( X \) with center \( x \in X \) and radius \( q \in \RR_{> 0} \).
\begin{enumerate-(1)}
\item \label{fct:basicsonultra-1}
Using~\eqref{eq:ultra}, it is easy to check that for every \( x,y,z \in X \)
\[ 
\text{if } d(x,y) > d(y,z) \text{ then } d(x,z) = d(x,y).
 \] 
 \item \label{fct:basicsonultra-2}
By part~\ref{fct:basicsonultra-1}, every two open balls \( B =  B_q(x) \) and \( B' =  B_{q'}(x') \) are either disjoint, or one of them is included in the other one. In particular, if \( q = q' \) then
either \( B = B' \) or else \( B \cap B' = \emptyset \), and two points of \( X \) have distance \( < q \) if and only if they belong to the same 
open ball with radius \( q \).
\item \label{fct:basicsonultra-3}
By part~\ref{fct:basicsonultra-2}, each \( B_q(x) \) is actually clopen, and if \( y \in B_q(x)  \) then \( B_q(y) = B_q(x) \). 
Thus if \( X \) has density character \( \kappa \) there are at most \( \kappa \)-many open balls with rational radius
(because it is enough to consider only balls centered in any fixed dense subset of \(X \) of size \( \kappa \)).
  \end{enumerate-(1)}
\end{fact}
Let \( \mathcal{L} = \{ R \} \cup \{ S_q \mid q \in \QQ_{> 0} \} \) be a first-order language with \( R \) a binary relation and \( S_q \) a unary relation (for \( q \in \QQ_{>0} \)). Given a Polish ultrametric space 
\( X = (X,d) \), let \( A_X = (B_X, R^X, (S^X_q)_{q \in \QQ_{>0}}) \) be the \( \mathcal{L} \)-structure defined by
\begin{align*}
B_X & = \{ B_q(x) \mid x \in X \text{ and } q \in \QQ_{>0} \} \\
R^X (B_1,B_2) & \iff B_1 \subseteq B_2 \\
S^X_q(B) & \iff \mathrm{diam}(B) < q.
\end{align*}
(Notice that since \( X \) is separable, the \( \mathcal{L} \)-structure \( A_X \) is countable by Fact~\ref{fct:basicsonultra}\ref{fct:basicsonultra-3}.) The map \( X \mapsto A_X \) is easily seen to be Borel: we claim that it 
also reduces isometry to isomorphism. 

Fix two Polish ultrametric spaces  \( X = (X,d_X) \) and \( Y = (Y,d_Y) \).
Clearly, if \( X \cong^i  Y \) then \( A_X \cong A_Y \). Conversely, let \( f \colon A_X \to A_Y \) be an isomorphism. Given \( x \in X \), let \( B^x_n = f(B_{\frac{1}{n}}(x)) \). Using the fact that \( f \) is an isomorphism and 
Fact~\ref{fct:basicsonultra}\ref{fct:basicsonultra-3}, we get that \( (B^x_n)_{n \in \Nat} \) is a \( \subseteq \)-decreasing sequence of clopen sets whose diameters converge to \( 0 \): thus \( \bigcap_{n \in \Nat} B^x_n \) 
contains a unique element \( y_x \). We claim that the map
\[ 
\varphi \colon X \to Y \qquad x \mapsto y_x
 \] 
is an isometry. We first check that \( \varphi \) preserves distances (whence it is also injective). Clearly, for every \( x\in X \) and \( B \in B_X \),
\begin{equation} \label{eq:f(B)}
x \in B \iff \varphi(x) \in f(B).
 \end{equation}
 and
 \begin{equation} \label{eq:diam}
\mathrm{diam}(B) = \mathrm{diam}(f(B)).
 \end{equation}
By Fact~\ref{fct:basicsonultra}\ref{fct:basicsonultra-2} two points \( x,x' \in X \) have distance \( < q \) if and only if they both belong to 
the same open ball \( B \) with radius \( q \): thus \( \varphi(x), \varphi(x') \in f(B) \) by~\eqref{eq:f(B)}, and since \( \mathrm{diam}(f(B)) = \mathrm{diam}(B) \leq q \), this implies that \( d_Y(\varphi(x),\varphi(x')) \leq q \). 
Similarly, reversing the above argument we get that if \( \varphi(x),\varphi(x') \) have distance \( < q \) then they both belong to the same open ball \( B \) with radius \( q \), so that \( x,x' \in f^{-1}(B) \) and 
\( d_X(x,x') \leq q \). This shows that \( d_X(x,x') = d_Y(\varphi(x), \varphi(x')) \).
It remains to show that \( \varphi \) is also surjective. This can be done by reversing the construction of \( \varphi \): given \( y \in Y \) 
check that \( (f^{-1}(B_{\frac{1}{n}}(y)))_{n \in \Nat}\) is a \( \subseteq \)-decreasing sequence of clopen set with vanishing diameter, and that the unique \( x \in \bigcap_{n \in \Nat} f^{-1}(B_{\frac{1}{n}}(y)) \) 
is such that \( \varphi(x) = y \).
\end{proof}

Gao and Kechris  showed that also the classifications given by Theorem~\ref{thm:upperboundultrametric} are optimal.

\begin{theorem}[Gao-Kechris, see~{\cite[Proposition 4.2 and Theorem 4.4]{Gao2003}}] \label{thm:lowerboundsfordiscrete}
\begin{enumerate-(1)}
\item \label{thm:lowerboundsfordiscrete-1}
The isomorphism relation \( \cong_{\mathsf{GRAPHS}} \) Borel reduces to isometry between discrete (hence also countable, zero-dimensional, and locally compact) 
Polish metric spaces.
\item \label{thm:lowerboundsfordiscrete-2}
The relation \( \cong_{\mathsf{GRAPHS}} \) Borel reduces to isometry between Polish ultrametric spaces.
\end{enumerate-(1)}
\end{theorem}

Although quite elementary, 
Gao-Kechris' proofs of the two parts of Theorem~\ref{thm:lowerboundsfordiscrete} use different (and incompatible) constructions, and in fact the complexity of the isometry relation on discrete Polish ultrametric spaces has 
been a longstanding open problem (see below) until~\cite{Camerlo:2015ar}: we will see that the solution to such problem provides in particular a (uniform) proof of both part~\ref{thm:lowerboundsfordiscrete-1} 
and~\ref{thm:lowerboundsfordiscrete-2}, so we omit the proof  of Theorem~\ref{thm:lowerboundsfordiscrete} here.

Theorem~\ref{thm:lowerboundsfordiscrete}\ref{thm:lowerboundsfordiscrete-1} also gives a lower bound for the complexity of the isometry relation on zero-dimensional Polish metric spaces: however, this lower bound is 
 not sharp.%
 \footnote{The problem of finding an optimal classification for zero-dimensional Polish metric spaces is still widely open --- see the Main Open Problem~\ref{problem:zerodim} and the comment after it.}
 Indeed, the map \( z \mapsto X_z \) from the proof of Theorem~\ref{thm:lowerboundallPolish} reducing \( E^Y_G \) to \( \cong^i \) has the interesting property that the topology of each \( X_z \) is closely related to 
 that of \( G \). 
 For example, if \( G \) is locally compact, then it admits a complete left-invariant metric (see e.g.~\cite[Lemma 6.5]{Gao2003}): it follows that \( \widehat{G} = G \), so that each of the spaces \( X_z \) is locally compact too.
Similarly, if \( G \) is 
 zero-dimensional and has a compatible \emph{complete} left-invariant metric (so that \( \widehat{G} = G \) again), then each \( X_z \) is zero-dimensional as well.
 This allows us to prove the 
 following.

\begin{theorem}[Clemens, see~{\cite[Corollary 19]{clemensisometry}}] \label{thm:lowerboundzerodim}
The isometry relation on zero-dimensional Polish metric spaces is not classifiable by countable structures.
\end{theorem}

\begin{proof}
Consider the \emph{density ideal} 
\[ 
I_d = \left \{ x \subseteq \NN \mathrel{\Big |} \lim_{n \to \infty} \frac{|\{ j \in x \mid j < n \}|}{n} = 0 \right \} 
\]
equipped with the group operation \( \triangle \) of symmetric difference. It is not difficult to show that \( (I_d, \triangle) \) can be given a complete left-invariant metric which turns it into a zero-dimensional Polish group
(see~\cite[Section 5]{clemensisometry}). 
Thus by (the proof of) Theorem~\ref{thm:lowerboundallPolish} and the observation preceding Theorem~\ref{thm:lowerboundzerodim}, for any continuous action of \( I_d \) on a Polish space \( Y \) we have that \( E^Y_{I_d} \) is Borel 
reducible to the isometry relation on zero-dimensional Polish metric spaces. As the action of \( I_d \) on the Cantor space \( \pre{\Nat}{2} \) via symmetric difference is turbulent (hence not classifiable by countable 
structures), the result follows. 
\end{proof}

\begin{remark}
It can be shown that locally compact groups do not admit turbulent actions: in fact, each orbit equivalence relation \( E^X_G \) induced by an action of a locally compact Polish group is essentially countable, and thus, 
in particular, classifiable by countable structures. Thus Clemens' technique cannot be used to obtain a non-classifiability result for locally compact Polish metric spaces analogous to Theorem~\ref{thm:lowerboundzerodim}.
\end{remark}

We have seen that isometry on compact metric spaces and isometry on arbitrary Polish metric spaces lie at the extremes of a wide range of possibilities in the structure of analytic equivalence relations under Borel reducibility, 
the former being Borel bi-reducible with \( \id(\RR) \) and the latter being bi-reducible with \( E^\infty_{ H([0;1]^\mathbb{N})} \), the most complicated orbit equivalence relation
(see Figure~\ref{fig:sumup}). It is thus natural to ask where it is located in such hierarchy the 
restriction of \( \cong^i \) to a natural mild enlargement of the class of compact metric spaces, namely \emph{locally compact} 
Polish metric spaces. Theorem~\ref{thm:lowerboundsfordiscrete}\ref{thm:lowerboundsfordiscrete-1} already 
shows that \( \cong_{\mathsf{GRAPHS}} \) is a lower bound for it, but to the best of our knowledge it is currently not known if such a lower bound is sharp --- see the Main Open Problem~\ref{problem:loccompact} and the discussion
after it. In the attempt of better understanding this classification problem, Gao and Kechris 
considered in~\cite{Gao2003} various subclasses  of locally compact Polish metric spaces and determined the complexity of the isometry relation on them. We summarize those results (without proofs) in the next theorem.%
\footnote{As acknowledged in~\cite{Gao2003}, the proof of Theorem~\ref{thm:subclassesofloccompact}\ref{thm:subclassesofloccompact-hjorth} is actually due to G.\ Hjorth.}

\begin{theorem}[Gao-Kechris] \label{thm:subclassesofloccompact}
\begin{enumerate-(1)}
\item 
The isometry relation on \emph{zero-dimensional} locally
compact Polish metric spaces is Borel bi-reducible with graph isomorphism~\cite[Theorem 4.3]{Gao2003}.
\item \label{thm:subclassesofloccompact-hjorth}
The isometry relation on \emph{connected} locally
compact Polish metric spaces is Borel bi-reducible with \( E_\infty \), the most complex (essentially) countable Borel equivalence relation. 
The same is true for \emph{Heine-Borel} or even for \emph{pseudo-connected} 
locally compact Polish metric spaces~\cite[Theorem 7.1]{Gao2003}. (Recall that pseudo-connected%
\footnote{See~\cite{Gao2003} for the definitions of pseudo-connected Polish metric spaces and pseudo-connected
components of locally compact Polish metric spaces.}
locally compact Polish metric spaces include both the connected and the Heine-Borel ones.)
\item \label{thm:subclassesofloccompact-hom}
The isometry relation on \emph{homogeneous pseudo-connected}
locally compact Polish metric spaces is smooth~\cite[Corollary 5.8]{Gao2003}.
(Recall that a metric space is homogenous if its isometry group acts transitively
on it.)
\end{enumerate-(1)}
\end{theorem}

The classification result given by Theorem~\ref{thm:subclassesofloccompact}\ref{thm:subclassesofloccompact-hom} has been improved by Clemens in~\cite{clemenshomogeneous}  by showing that the complexity of isometry on all homogeneous locally compact Polish metric spaces  becomes quite higher if we remove the additional condition of pseudo-connectedness. 

\begin{theorem}[Clemens, see~{\cite[Corollaries 6.2 and 6.3]{clemenshomogeneous}}] \label{thm:clemenshomogeneous}
Isometry between homogeneous \emph{discrete} Polish metric spaces is Borel bi-reducible
with \( \cong_{\mathsf{GRAPHS}} \). Thus \( \cong_{\mathsf{GRAPHS}} \) is Borel reducible to isometry between homogeneous
\emph{locally compact} Polish metric spaces.
\end{theorem}

It is not known whether the above lower bound for homogeneous locally compact Polish metric spaces is sharp. 
In contrast, in~\cite{clemenshomogeneous} it is further shown that the complexity of isometry drops down again if 
we consider the smaller class of \emph{ultrahomogeneous} locally compact Polish metric spaces. (Recall that a metric space is ultrahomogeneous if every partial isometry between finite subsets of it can be extended 
to an isometry of the whole space.) Let \( E_{ctble} \) be the equivalence relation on \( \pre{\Nat}{\RR} \) defined by
\[ 
(x_i)_{i \in \Nat} \mathrel{E_{ctble}} (y_i)_{i \in \Nat} \iff \{ x_i \mid i \in \Nat \} = \{ y_i \mid i \in \Nat \},
 \] 
that is two sequences of reals \( (x_i)_{i \in \Nat} \) and \( (y_i)_{i \in \Nat} \) are \( E_{ctble} \)-equivalent if and only if they enumerate 
(possibly with repetitions) the same countable sets of reals. It can be shown that
\[ 
E_\infty <_B E_{ctble} <_B \cong_{\mathsf{GRAPHS}}.
 \]

\begin{theorem}[Clemens, see~{\cite[Theorems 7.2 and 7.6]{clemenshomogeneous}}]
Isometry between ultrahomogeneous locally compact Polish metric spaces is Borel bi-reducible
with \( E_{ctble} \). In fact, the same is true already for ultrahomogeneous discrete Polish metric spaces.
\end{theorem}

It is maybe worth mentioning the following related result on (ultra)homogeneous Polish \emph{ultrametric} spaces.

\begin{theorem}[Gao-Kechris, see~{\cite[Theorem 4.5]{Gao2003}}]
The isometry relation on ultrahomogeneous discrete
Polish ultrametric spaces is Borel
bi-reducible with \( E_{ctbie} \). The same is true for isometry between homogeneous Polish ultrametric spaces.
\end{theorem}

Gao and Kechris devoted the entire Chapter 8 of their monograph~\cite{Gao2003} to another subclass of locally compact Polish metric spaces, namely the ultrametric ones. They were able to show that each such space \( X = (X,d) \) 
can be decomposed  in a particularly nice way using pseudo-connected components. More precisely, one can form a derived space \( D(X) \) by choosing a point from each of the countably many pseudo-components of \( X \): 
it turns out that \( D(X) \)
is a discrete, closed subspace of \( X \) which
contains the complete information about the metric relations among pseudo-components
of \( X \) (see~\cite[Proposition 8.1(ii)]{Gao2003}. On the other hand, by~\cite[Proposition 8.1(i)]{Gao2003}   there are only three possibilities for an arbitrary pseudo-component \( C \) of \( X \):
\begin{enumerate-(a)}
\item
\( C \) is compact; this happens when its diameter is attained by a pair of points.
\item
\( C \) is Heine-Borel; this is the case where 
\( X = C \) is pseudo-connected.
\item
Otherwise: it must then happen that  \( \mathrm{diam}(C) \) is finite and it is
not attained by any pair of points of \( C \). In this case one can define \( d' \) on \( C \) by
\[
d' = \frac{d}{\mathrm{diam}(C) - d}.
\]
Then \( ( C , d' ) \) is ultrametric and Heine-Borel, and it contains all the information about the isometry type of \( (C,d) \).
\end{enumerate-(a)}
These observations essentially reduce the problem of understanding the complexity of the isometry relation on the class of all locally compact Polish ultrametric spaces to the analogous problem for two%
\footnote{Recall that we already know that \emph{compact} Polish (ultra)metric spaces are concretely classifiable by Theorem~\ref{thm:gromov}.}
particular subclasses  of it: Heine-Borel and discrete.
 The first one is already dealt with in~\cite[Section 8.B]{Gao2003} (compare the following result with Theorem~\ref{thm:subclassesofloccompact}\ref{thm:subclassesofloccompact-hjorth} concerning arbitrary Heine-Borel Polish metric spaces):

\begin{theorem}[Gao-Kechris, see~{\cite[Theorem 8.2]{Gao2003}}] \label{thm:heineborelultra}
The isometry relation on Heine-Borel Polish ultrametric spaces is Borel bi-reducible
with \( E_0 \).
\end{theorem}

As for the other class, namely discrete Polish ultrametric spaces, Gao and Kechris isolated in~\cite[Section 8.C]{Gao2003} some lower bounds (although they could not show that they were optimal). 
In fact, they showed in~\cite[Theorem 8.10]{Gao2003} that the following 
equivalence relations are pairwise Borel bi-reducible:
\begin{enumerate-(1)}
\item
Isomorphism on \( \mathrm{WF} \), the set of well-founded trees on \( \Nat \) (see Notation~\ref{not:trees}).
\item \label{lowboundsultra-2}
Isometry between discrete closed subspaces of the Baire space \( \pre{\Nat}{\Nat} \).
\item \label{lowboundsultra-3}
Isometry between locally compact closed subspaces of \( \pre{\Nat}{\Nat} \).
\end{enumerate-(1)}
Since the metric spaces in items~\ref{lowboundsultra-2} and~\ref{lowboundsultra-3} are clearly Polish ultrametric spaces, the isomorphism relation between well-founded trees on \(\Nat\) (which is quite well-understood) 
becomes a lower bound for the complexity of isometry between 
discrete (or locally compact) Polish ultrametric spaces. Another lower bound, namely isomorphism between the so-called \emph{reverse trees}, has been proposed by Clemens in his preprint~\cite{ClemensPreprint}. 
Unfortunately, it was totally unclear whether these lower bounds were sharp or not, 
so the general problem of determining the exact complexity of isometry on discrete or locally compact Polish ultrametric spaces remained open.

A different approach to solve this problem was later suggested by Gao and Shao in~\cite{GaoShao2011}.

\begin{notation}
Let \( X = (X,d) \) be a metric space and \( y \in X \). We let
\[ 
R_X(y) = \{ d(y,x) \mid x \in X, x \neq y \}
 \] 
be the set of nonzero distances realized by \( y \) in \( X \), and
\[ 
R(X) = \bigcup_{y \in X} R_X(y) = \{ d(y,x) \mid x,y \in X, x \neq y \}
 \] 
be the set of all nonzero distances realized in \( X \).
\end{notation} 

\begin{lemma} \label{lemma:dense}
Let \( X = (X,d) \) be an ultrametric space and \( Y \) be any dense subspace of \( X \). Then
\[ 
R(X) = R(Y).
 \] 
\end{lemma}

\begin{proof}
The inclusion \( R(Y) \subseteq R(X) \) is obvious, so fix distinct \( x_0,x_1 \in X \). Using the density of \( Y \subseteq X \), choose \( y_0, y_1 \in Y \) such that \( d(y_i,x_i)  < d(x_0,x_1) \) for \( i = 0,1 \).
Then by applying (twice) Fact~\ref{fct:basicsonultra}\ref{fct:basicsonultra-1} we get
\[ 
d(x_0,x_1) = d(y_0,x_1) = d(y_0,y_1),
 \] 
whence \( d(x_0,x_1) \in R(Y) \). This shows \( R(X) \subseteq R(Y) \) and concludes the proof.
\end{proof}

In particular, if an ultrametric space \( X \) has density character \( \kappa \), then \( R(X) \) is a subset of \( \RR_{> 0} \) of size \( \leq \kappa \). Conversely, let \( D \subseteq \RR_{> 0} \) be of size \( \leq \kappa \): then the space \( U_D = (U_D,d_D) \) define by setting \( U_D = D \times \kappa \) and
\[ 
d_D((r,\alpha),(q,\beta)) = 
\begin{cases}
0 & \text{if } (r,\alpha) = (q,\beta) \\
\max \{ r , q \} & \text{otherwise}
\end{cases}
 \] 
is a discrete (hence complete) ultrametric space of density character \( \kappa \) such that \( R(U_D) = D \). In particular, letting \( \US_\Sep \subseteq \MS_\Sep \) be the collection of (codes for) Polish ultrametric spaces one has
\[ 
\{ R(X) \mid X \in \US_\Sep \} = \{ D \subseteq \RR_{>0} \mid D \text{ is (at most) countable} \}.
 \] 
 For notational simplicity, we denote by \( \mathcal{D}_\Sep \) the collection of (at most) countable subsets of \( \RR_{> 0} \) considered in the previous equation.

The idea of Gao and Shao is to study the restriction of the isometry relation \( \cong^i \), which in the sequel will be denoted by \( \cong^i_D \), to the hyperspace
\[ 
\US^D_\Sep = \{ X \in \US_\Sep \mid R(X) \subseteq D \}
 \] 
for a fixed in advance set of distances \( D \in \mathcal{D}_\Sep \) (using Lemma~\ref{lemma:dense} and the Borel selectors \( \psi_n \) 
in the obvious way, such an hyperspace is easily seen to be a Borel subset of \( \MS_\Sep \), hence it is a standard Borel space).
This is strictly related to our problem, since if \( D \) is bounded away from \( 0 \) (that is, if there is \( r > 0 \) such that \( D \subseteq [r;+ \infty) \)) then \( \US^D_\Sep \) consists only of discrete (hence locally compact) Polish ultrametric spaces.

In~\cite[Section 8]{GaoShao2011} it is observed that if \( D \) is \emph{not} bounded away from \( 0 \) (i.e.\ if \( 0 \) is an accumulation point of \( D \)), 
then \( \cong^i_D \) is Borel bi-reducible with \( \cong_{\mathsf{GRAPHS}} \). In fact, the spaces constructed by Gao and Kechris in their proof of Theorem~\ref{thm:lowerboundsfordiscrete}\ref{thm:lowerboundsfordiscrete-2} used nonzero distances in the set \( R(\pre{\Nat}{\Nat}) =  \{ 2^{-n} \mid n \in \Nat \} \), so that the result is true for all \( D \supseteq R(\pre{\Nat}{\Nat}) \); the general case is then obtained using in the obvious way the following ``transfer'' result.

\begin{lemma}[Gao-Shao, see~{\cite[Theorem 8.2]{GaoShao2011}}] \label{lemma:transfer}
Let \( D_1 , D_2 \in \mathcal{D}_\Sep \). If there exists an order preserving injection \( \rho \colon D_1 \cup \{ 0 \} \to  D_2 \cup \{ 0 \} \) with
\( \rho(0) = 0 \) that is continuous at \( 0 \), then \( {\cong^i_{D_1}} \leq_B {\cong^i_{D_2}} \).
Consequently, if there
is an order preserving bijection \( \rho \colon D_1 \cup \{ 0 \} \to D_2 \cup \{ 0 \} \) with \( \rho(0) = 0 \) that is both open and continuous at \( 0 \), then \( {\cong^i_{D_1}} \sim_B {\cong^i_{D_2}} \).
\end{lemma}

The proof of Lemma~\ref{lemma:transfer} is easy: associate to any given \( X = (X,d) \in \US^{D_1}_\Sep \) the space \( X' = (X, \rho \circ d) \in \US^{D_2}_\Sep \), and check that the Borel map \( X \mapsto X' \) 
is the desired reduction. As we will see later, Lemma~\ref{lemma:transfer} can be improved by removing the condition on the continuity of the map \(\rho\) at \( 0 \): this is a nontrivial matter, as without such condition the 
map \( X \mapsto X' \) considered above needs not to produce \emph{Polish} ultrametric spaces, so a different argument must be employed.

The main problem left open by~\cite[Section 8]{GaoShao2011} is thus what happens to the isometry relation \( \cong^i_D \) \emph{for \( D \) bounded
 away from \( 0 \)}. (Recall that by Lemma~\ref{lemma:transfer} the complexity of such a relation depends only on the order type of \( D \), and each of these \( \cong^i_D \) is a lower bound for the complexity of the 
 isometry relation on discrete (hence also on locally compact) Polish ultrametric spaces.)
This problem has been recently solved in~\cite{Camerlo:2015ar}. Quite surprisingly, it turns out that the complexity of \( \cong^i_D \) \emph{for an arbitrary \( D \in \mathcal{D}_\Sep \)} depends just on the order type of \( D \), rather than on its topology as a subset of \( \RR \); in particular, whether \( 0 \) is a limit point of \( D \) or not is totally irrelevant (see Lemma~\ref{lemma:transferimproved}).

Let us fix some terminology and notation. We say that \( D \in \mathcal{D}_\Sep \) is \emph{ill-founded} if it contains an infinite strictly decreasing (with respect to the usual ordering on \( \RR \)) sequence of reals; otherwise we say that \( D \) is \emph{well-founded}. In the latter case, the order type of \( D \) is isomorphic to a (unique) countable ordinal, that will be denoted by \( \alpha_D \).

\begin{theorem}[Camerlo-Marcone-Motto Ros] \label{thm:maincammarmot}
\begin{enumerate-(1)}
\item \label{thm:maincammarmot-1} 
If \( D \in \mathcal{D}_\Sep \) is ill-founded, then \( \cong^i_D \) is Borel bi-reducible with \( \cong_{\mathsf{GRAPHS}} \) (\cite[Corollary 5.3]{Camerlo:2015ar}).
\item \label{thm:maincammarmot-2}
If \( D \in \mathcal{D}_\Sep\) is well-founded, then \( \cong^i_D \) is a Borel equivalence relation. Moreover, if both \( D , D' \in \mathcal{D}_\Sep \) are well-founded then
\[ 
{\cong^i_D} \leq_B {\cong^i_{D'}} \iff \alpha_D \leq \alpha_{D'},
 \] 
and for each Borel equivalence relation \( E \) classifiable by countable structures \( E \leq_B {\cong^i_D} \) for some well-founded \( D \in \mathcal{D}_\Sep \). Thus the isometry relations \( \cong^i_D \) for \( D \in \mathcal{D}_\Sep \) well-founded form an \(\omega_1 \)-chain of Borel equivalence relations cofinal among the Borel equivalence relations classifiable by countable structures 
(\cite[Theorem 5.15]{Camerlo:2015ar}).
\end{enumerate-(1)}
\end{theorem}

The original proof of Theorem~\ref{thm:maincammarmot}\ref{thm:maincammarmot-1} from~\cite{Camerlo:2015ar} used a new construction which, unfortunately, works only when \( D \in \mathcal{D}_\Sep \) is bounded away from \( 0 \).
This is of course not restrictive, 
as the case of \( D \in \mathcal{D}_\Sep \) with \( 0 \)  as an accumulation point was already treated 
in~\cite[Section 8]{GaoShao2011}; however, it is a bit unsatisfactory because it may suggest that, if not the result, at least its proof still depends on the topology of \( D \) (as a subset of \( \RR \)) rather than just on 
its order type. Nevertheless, in Section~\ref{sec:treesisometry} we will provide a modification (and generalization) of the argument from~\cite{Camerlo:2015ar} which will allow us to deal with both cases 
simultaneously,  thus eliminating the unsatisfactory case distinction discussed above; for this reason, we postpone the proof of Theorem~\ref{thm:maincammarmot} to the end of Section~\ref{sec:treesisometry}.

By taking a bounded away from \( 0 \) ill-founded \( D \in \mathcal{D}_\Sep \) in (the proof of) Theorem~\ref{thm:maincammarmot}\ref{thm:maincammarmot-1}, one obtains a uniform proof of both part~\ref{thm:lowerboundsfordiscrete-1} and part~\ref{thm:lowerboundsfordiscrete-2} of Theorem~\ref{thm:lowerboundsfordiscrete}. The same argument (combined with Theorem~\ref{thm:upperboundultrametric}\ref{thm:upperboundultrametric-2}) 
allows us to compute the precise complexity of the isometry relation on discrete Polish ultrametric spaces, on locally compact Polish ultrametric spaces, as well as on many other natural classes of Polish ultrametric spaces.

\begin{corollary}[Camerlo-Marcone-Motto Ros, see~{\cite[Corollary 5.3]{Camerlo:2015ar}}  and the comment after it] \label{cor:maincammarmot} 
The restrictions of the isometry relation to the following subclasses of \( \US_\Sep \) are all Borel bi-reducible with \( \cong_{\mathsf{GRAPHS}} \), the most complex analytic equivalence relation among the ones which are classifiable by countable structures (see Figure~\ref{fig:sumup}):
\begin{itemizenew}
\item
\emph{countable} Polish ultrametric spaces;
\item
\emph{discrete} Polish ultrametric spaces;
\item
\emph{locally compact} Polish ultrametric spaces;
\item
\emph{\(\sigma\)-compact} Polish ultrametric spaces.
\end{itemizenew}
\end{corollary}

Let us observe that Theorem~\ref{thm:maincammarmot} and Corollary~\ref{cor:maincammarmot} also show that both the lower bounds for the complexity of discrete/locally compact Polish 
ultrametric spaces isolated by Gao-Kechris and Clemens (see the paragraph after Theorem~\ref{thm:heineborelultra}) were not sharp. In fact:
\begin{itemizenew}
\item
by the subsequent Proposition~\ref{prop:isoWF}, the isomorphism relation on the collection \( \mathrm{WF} \) of well-founded trees on \(\Nat\) is not Borel bi-reducible with any 
equivalence relation with standard Borel domain; since \( \cong^i \) on discrete/locally compact Polish ultrametric spaces is instead Borel bi-reducible with such an equivalence relation 
(e.g.\ \( \cong_{\mathsf{GRAPHS}} \)), this shows that the isomorphism relation on \( \mathrm{WF} \) is strictly simpler than the latter isometry relations;
\item
it is not hard to show that the isomorphism relation on reverse trees is Borel bi-reducible with \( \cong^i_D \), where \( D \in \mathcal{D}_\Sep \) is some/any well-founded set of order type \(\omega\): since the latter is 
a Borel equivalence relation by Theorem~\ref{thm:maincammarmot}\ref{thm:maincammarmot-2}, we obtain that such an isomorphism relation is necessarily strictly simpler than the isometry on discrete/locally compact 
Polish ultrametric spaces, which by Theorem~\ref{thm:maincammarmot}\ref{thm:maincammarmot-1} is proper analytic.
\end{itemizenew}

Finally, Theorem~\ref{thm:maincammarmot} yields the following improvement of Lemma~\ref{lemma:transfer}.
\begin{lemma}[Camerlo-Marcone-Motto Ros, see~{\cite[Corollary 5.11]{Camerlo:2015ar}}] \label{lemma:transferimproved}
Let \( D_1 , D_2 \in \mathcal{D}_\Sep \). If there exists an order preserving injection \( \rho \colon D_1  \to  D_2 \), then \( {\cong^i_{D_1}} \leq_B {\cong^i_{D_2}} \).
Consequently, if \( D_1  \) and \( D_2 \) are order-isomorphic, then \( {\cong^i_{D_1}} \sim_B {\cong^i_{D_2}} \).
\end{lemma}

\begin{proof}
If \( D_2 \) is ill-founded, then
\[ 
{\cong^i_{D_1}} \leq_B {\cong_{\mathsf{GRAPHS}}} \leq_B {\cong^i_{D_2}}
 \] 
by Theorems~\ref{thm:upperboundultrametric}\ref{thm:upperboundultrametric-2} and~\ref{thm:maincammarmot}\ref{thm:maincammarmot-1}. If instead \( D_2 \) is well-founded and there is an order preserving 
injection
\( \rho \colon D_1 \to D_2 \), then \( D_1 \) is well-founded too, and \(\alpha_{D_1} \leq \alpha_{D_2} \): thus the result follows from Theorem~\ref{thm:maincammarmot}\ref{thm:maincammarmot-2}.
\end{proof}

Let us conclude this section by mentioning again the two main open problems in this area of research.

\begin{mainproblem} \label{problem:zerodim}
Determine the exact complexity (in the hierarchy of Borel
reducibility) of isometry between \emph{zero-dimensional} Polish metric spaces.
\end{mainproblem}

Recall that by Theorems~\ref{thm:lowerboundsfordiscrete}\ref{thm:lowerboundsfordiscrete-1} and~\ref{thm:lowerboundzerodim}, the above isometry relation is \emph{strictly} more complicated than \( \cong_{\mathsf{GRAPHS}} \): 
this lead Gao, Kechris and Clemens to conjecture that it is Borel bi-reducible with \( \cong^i \) on the whole \( \MS_\Sep \), that is with the most complex orbit equivalence relation.

\begin{mainproblem} \label{problem:loccompact}
Determine the exact complexity (again with respect to Borel reducibility) of the isometry relation on \emph{locally compact} Polish metric spaces.
\end{mainproblem}

By Theorem~\ref{thm:lowerboundsfordiscrete}\ref{thm:lowerboundsfordiscrete-1}, the isomorphism relation \( \cong_{\mathsf{GRAPHS}} \) is a lower bound for this isometry relation too, but in this case it may be sharp: 
in fact, Gao and Kechris explicitly conjectured in~\cite{Gao2003} that this is the case. In this direction, Hjorth proved that isometry of locally compact Polish metric
spaces is at least reducible to graph isomorphism by an \emph{absolutely \( \mathbf{\Delta}^1_2 \)} function: this provides
strong evidence for the truth of the conjecture (although there are examples%
\footnote{One such example is given by \( \cong_{\mathsf{GRAPHS}} \), which is absolutely  \( \mathbf{\Delta}^1_2 \) reducible to isometry on the set \( \mathrm{WF} \) of well-founded trees on \(\Nat\), but is not Borel reducible to it by Proposition~\ref{prop:isoWF}.} 
of analytic equivalence relations which are absolutely \( \mathbf{\Delta}^1_2 \) reducible to another one but not Borel reducible to it).
A related problem, still widely open, is that of determining the exact complexity of isometry between \emph{homogeneous} locally compact Polish metric spaces, for which \( \cong_{\mathsf{GRAPHS}} \) 
is again a lower bound by a result of Clemens (see Theorem~\ref{thm:clemenshomogeneous}).

\section{Trees and isometry} \label{sec:treesisometry}

Given an ordinal \(\alpha \in \On \) and an infinite cardinal \( \kappa \), let \( \pre{<\alpha}{\kappa} \) be the collection of all sequences of the form \( u \colon \beta \to \kappa \) 
for some \( \beta < \alpha \) (recall that an ordinal \(\beta\) is identified with the set \( \{ \gamma \in \On \mid \gamma < \beta \} \) of its predecessors, so that \( 0 \) is the empty set). The set \( \pre{<\alpha}{\kappa} \) is ordered by inclusion, i.e.\ for \( u,v \in \pre{<\alpha}{\kappa} \) we set
\[ 
u \subseteq v \iff u \text{ is an initial segment of } v.
 \] 
As usual, for \( u \in \pre{<\alpha}{\kappa} \) we call the ordinal \( \dom(u) \) \emph{length of \( u \)}, and denote it by \( \lh(u) \). The sequences \( u,v \in \pre{<\alpha}{\kappa} \) are \emph{compatible} if \( u \subseteq v \) or \( v \subseteq u \), otherwise they are \emph{incompatible}. Given \( u \in \pre{<\alpha}{\kappa} \) and \( \beta \leq \lh(u) \), we denote by \( u \restriction \beta \) the \emph{restriction} of \( u \) to \(\beta\), i.e.\ the unique \( v \subseteq u \) such that \( \lh(v) = \beta \).

\begin{defin} \label{def:alpha-tree}
Let \( \alpha \in \On \) and \( \kappa \) be an infinite cardinal. An \emph{\(\alpha\)-tree of size \( \kappa \)} is a subset \( T \) of \( \pre{<\alpha}{\kappa} \) (still ordered by inclusion) satisfying the following conditions: 
\begin{enumerate-(1)}
\item
\( T \) is \( \subseteq \)-downward closed;
\item \label{cond:2}
\( |T| = \kappa \).
\end{enumerate-(1)}
The collection of all \( \alpha \)-trees of size \( \kappa \) is denoted by \( \T{\alpha}{\kappa} \).
\end{defin}

Notice that for all distinct \( u,v \in \pre{< \alpha}{\kappa} \) it is always the case that there is \( \beta \leq \lh(u),\lh(v) \) such that \( u \restriction \beta = v \restriction \beta \) and either one of \( u(\beta) , v(\beta) \) is undefined or 
else \( u(\beta) \neq v(\beta) \); indeed, it is enough to set 
\[ 
\beta = \min \{ \gamma \leq \lh(u),\lh(v) \mid \gamma = \lh(u) \text{ or } \gamma = \lh(v) \text{ or } u(\gamma) \neq v(\gamma) \} . 
\]
We denote by \( u \sqcap v \) the sequence 
\( u \restriction \beta = v \restriction \beta \) (with \(\beta\) as above): this is the longest common initial segment of \( u \) and \( v \).

A node \( u \) of an \(\alpha\)-tree \( T \) is called \emph{terminal} if there is no proper extension of \( u \) in \( T \). An \emph{immediate predecessor of \( u \)} is a 
node \( v \in T \) such that there is no intermediate \( v \subsetneq w \subsetneq u \). Notice that \( u \) has an immediate predecessor if and only if \( \lh(u) \) is a 
successor ordinal \( \beta +1 \), and in this case its unique immediate predecessor is \( u \restriction \beta \). An \(\alpha\)-tree is \emph{pruned} if for every \( u \in T \) 
and \( \beta < \alpha \) there is \( v \in T \) such that \( \lh(v) = \beta \) and \( v \) is compatible with \( u \). Notice that when \( \alpha > \omega \) this is stronger than 
requiring that \( T \) has no terminal node.

The following is the natural notion of isomorphism between \(\alpha\)-trees of size \( \kappa \).

\begin{defin}
Let \( T,S \in \T{\alpha}{\kappa}\). An \emph{isomorphism} of \( T \) into \( S \) is a bijective map \( \varphi \colon T \to S \) which respect the inclusion relation, i.e.\ such that for all \( u,v \in T \)
\[ 
 u \subseteq v \iff \varphi(u) \subseteq \varphi(v).
 \] 
We denote by \( \cong \)  the relation of isomorphism on \( \T{\alpha}{\kappa} \). 
\end{defin}

A simple induction on the length of sequences in \( T \) shows that any isomorphism \( \varphi \colon T \to S \) automatically preserves their length, i.e.\ for every \( u \in T \)
\begin{equation} \label{eq:isomorphismspreservelengths} 
\lh(\varphi(u)) = \lh(u).
 \end{equation}

Let now \( \alpha \) be countable%
\footnote{The countability of \(\alpha\) is a necessary condition here because since \( \RR \) is separable there are no uncountable strictly decreasing sequences of reals.} 
 and \( D \subseteq \RR_{> 0} \) be such that \( D \) contains a strictly decreasing sequence \( (r_\beta)_{\beta< \alpha} \). Then we can turn any tree \( T \in \T{\alpha}{\kappa} \) into a complete ultrametric space \( X_T \) of density character \( \kappa \) and nonzero distances in \( D \) as follows. Equip \( T \) with the distance \( d_T \) given by
\[ 
d_T(u,v) = 
\begin{cases}
r_{\lh(u \sqcap v)} & \text{if } u \neq v \\
0 & \text{otherwise}.
\end{cases}
 \] 
This is clearly a well-defined ultrametric on \( T \), and hence we can let \( X_T = (X_T, d_T) \) be the completion of the ultrametric space \( (T,d_T) \). It is clear that \( T \) is dense in \( X_T \) and that it coincides with the collection of isolated points of \( X_T \) (so that \( X_T \) has density character \( \kappa \) by Definition~\ref{def:alpha-tree}\ref{cond:2}); indeed each \( u \in T \) is isolated by the open sphere \( B_{r_{\lh(u)}}(u) \), and each \( x \in X_T \setminus T \) is by definition a limit point. Furthermore, if the chosen sequence \( (r_\beta)_{\beta<\alpha}  \) is bounded away from \( 0 \), or if there is \( \beta < \alpha \) such that \( \lh(u) \leq \beta \) for every \( u \in T \) (which is always the case if \(\alpha\) is a successor ordinal), 
then \( X_T = T \) is a discrete space. In the remaining case, \( X_T \) may be identified with \( T \cup [T] \), where \( [T] \) is the body of \( T \) defined by
\begin{equation} \label{eq:body}
[T] = \{ f \colon \alpha  \to \kappa \mid f \restriction \beta  \in T \text{ for every } \beta < \alpha   \}.
 \end{equation}

\begin{remark} \label{rmk:crucial2}
By definition of the distance \( d_T \), we have that for distinct \( u,v \in T \)
\[ 
u \subseteq v \iff u \sqcap v = u \iff d_T(u,v) = r_{\lh(u)}.
 \] 
\end{remark}

\begin{remark} \label{rmk:crucial}
It is important to notice the following properties of the metric \( d_T \). Recall that given \( u \in T \) we let 
\[ 
R_T(u) = \{ r \in D \mid \exists v \in T \, (d_T(u,v) = r \}
 \] 
be the set of nonzero distances realized by \( u \) in \( T \). Then
\begin{itemize}
\item
\( \{ r_\beta \mid \beta < \lh(u) \}  \subseteq R_T(u) \subseteq \{ r_\beta \mid \beta \leq \lh(u) \} \);
\item
\( r_{\lh(u)}  \in R_T(u) \) if and only if \( u \) is not terminal in \( T \).
\end{itemize}
Indeed 
\[ 
d_T(u, u \restriction \beta) = r_{\lh(u \restriction \beta)} = r_\beta
 \] 
  for every \( \beta < \lh(u) \),
\( d_T(u,v) > r_{\lh(u)} \) if \( u \not\subseteq v \), while \( d_T(u,v) = r_{\lh(u)} \) if (and only if) \( u \subsetneq v \). 
\end{remark}

\begin{theorem} \label{thm:red}
For every \( T,S \in \T{\alpha}{\kappa} \)
\[ 
T \cong S \iff X_T \cong^i X_S.
 \] 
\end{theorem}

\begin{proof}
One direction is easy: if \( \varphi \colon T \to S\) is an isomorphism, then it preserves also the intersection operation \( \sqcap \), and thus \( \varphi \) is also an 
isometry between \( (T,d_T) \) and \( (S,d_S) \) that can be extended to an isometry \( \psi \) between \( X_T \) and \( X_S \).

Conversely, let \( \psi \colon X_T \to X_S \) be an isometry. Since \( \psi \) must preserve isolated points (in both directions), its restriction 
\( \varphi = \psi \restriction T \) is an isometry between \( (T,d_T) \) and \( (S,d_S) \).
 Unfortunately, \( \varphi \) needs not to be an isomorphism between \( \alpha \)-trees: if e.g.\ \( T = S \) and \( u \) is a terminal node of \( T \) with an immediate 
 predecessor \( v \), then the map switching \( u \) and \( v \) and fixing all other points of \( T \) is an isometry of \( (T,d_T) \) into itself, but it is not an 
 automorphism of \( T \) as an \(  \alpha \)-tree. However, we will now show that this is the unique possible obstruction for \( \varphi \) being an isomorphism.

 Call a pair \( ( v_0, v_1 ) \) a \emph{\(\varphi\)-switching pair} if the 
 following conditions hold:
\begin{itemize}
\item
\( v_1 \) is a terminal node of \( T \) and \( v_0 \) is an immediate predecessor of it;
\item
\( \varphi(v_0) \) is a terminal node of \( S \) and \( \varphi(v_1) \) is an immediate predecessor of it;
\item
\( \lh(v_i) = \lh(\varphi(v_{1-i})) \) for \( i = 0,1 \).
\end{itemize}
Notice that \( ( v_0,v_1 ) \) is a \(\varphi\)-switching pair if and only if \( ( \varphi(v_1), \varphi(v_0) ) \) is a \( \varphi^{-1} \)-switching pair.

\begin{claim} \label{claim:switsching}
Let \( u \in T \) be such that \( \lh(\varphi(u)) \neq \lh (u) \). Then \( u \) belongs to a \(\varphi\)-switching pair.
\end{claim}

\begin{proof}[Proof of the claim]
Suppose first that \( u \) is a terminal node, and let \( \gamma = \lh(u) < \alpha \). Then \( R_T(u) = \{ r_\beta \mid \beta < \gamma \} \) by Remark~\ref{rmk:crucial}. Since \( \varphi \) preserves distances realized by points, we also have 
\begin{equation} \label{eq:ru1}  
R_S(\varphi(u)) = \{ r_\beta \mid \beta < \gamma \}.
\end{equation} 
Since we assumed \( \lh(\varphi(u)) \neq \lh (u) \), by Remark~\ref{rmk:crucial} again and~\eqref{eq:ru1} we have \( \lh(\varphi(u)) < \lh(u) = \gamma \) and hence
\begin{equation} \label{eq:ru2}
R_S(\varphi(u)) = \{ r_\beta \mid \beta \leq \lh(\varphi(u)) \}.
\end{equation}
Combining~\eqref{eq:ru1},~\eqref{eq:ru2} and Remark~\ref{rmk:crucial}, we thus have that 
\begin{itemize}
\item
\( \gamma = \lh(u) = \lh(\varphi(u))+1 \);
\item
\( \varphi(u) \) is not a terminal node in \( S \).
\end{itemize}
Let \( \delta = \lh(\varphi(u)) \) (so that \( \gamma = \delta+1 \)), and let
 \( v = u \restriction \delta \), so that \( v \) is the immediate predecessor of \( u \): we claim that \( ( v,u ) \) is a \(\varphi\)-switching pair. Indeed, notice that by Remark~\ref{rmk:crucial} and~\eqref{eq:ru1}--\eqref{eq:ru2}
\begin{equation} \label{eq:rv}
R_T(v) = \{ r_\beta \mid \beta \leq \delta \} = \{ r_\beta \mid \beta \leq \lh(\varphi(u)) \} = R_S(\varphi(u)) = \{ r_\beta \mid \beta < \gamma \},
 \end{equation}
 so that \( R_S(\varphi(v)) = \{ r_\beta \mid \beta < \gamma \} \) as well. Since \( d_S(\varphi(u), \varphi(v)) = d_T(u,v) = r_{\lh(v)} = r_{\lh(\varphi(u))} \), 
we have \( \varphi(u) \subsetneq \varphi(v) \) by Remark~\ref{rmk:crucial2}. From this,~\eqref{eq:rv} and Remark~\ref{rmk:crucial} it follows that \( \lh(\varphi(v)) = \gamma \), \( \varphi(v) \) is a terminal node of \( S \), and \( \varphi(u) \) is its immediate predecessor. This concludes the proof of the fact that \( (v,u) \) is a \(\varphi\)-switching pair.
 
Suppose now that \( u \) is not a terminal node, so that \( R_T(u) = \{ r_\beta \mid \beta \leq \lh(u) \} \) by Remark~\ref{rmk:crucial}. 
Since \(\varphi\) preserves realized distances, we also have \( R_S(\varphi( u )) = \{ r_\beta \mid \beta \leq \lh(u) \} \), and since \( \lh(\varphi(u))\neq \lh(u) \) by hypothesis, by Remark~\ref{rmk:crucial} again we have 
that \( R_S(\varphi( u)) = \{ r_\beta \mid \beta < \lh(\varphi(u)) \} \) and thus \( \varphi(u) \) is a terminal node of \( S \). Switching the role of \( T \) and \( S \) and replacing \( \varphi \) with \( \varphi^{-1} \) in the first part of the proof, we obtain that 
\( \varphi(u) \) belongs to a \( \varphi^{-1} \)-switching pair \( ( w_0,w_1 ) \), so that \( u = \varphi^{-1}(\varphi(u)) \) belongs to the \( \varphi \)-switching pair \( ( \varphi^{-1}(w_1), \varphi^{-1}(w_0) ) \) and we are done again.
\end{proof}
Now define \( \varphi' \colon T \to S \) as follows. If \( u \in T \) does not belong to any \( \varphi \)-switching pair, then set \( \varphi'(u) = \varphi(u) \). 
If instead \( ( v_0,v_1 ) \) is a \(\varphi\)-switching pair, then set \( \varphi'(v_i) = \varphi(v_{1-i}) \) for \( i = 0,1 \).  
Notice that since the immediate predecessor of a node, if it exists, is unique and since \( \varphi \) is bijective, we get that distinct \(\varphi\)-switching pairs are necessarily disjoint: therefore the map \( \varphi' \) is a 
well-defined bijection between \( T \) and \( S \). Moreover, by Claim~\ref{claim:switsching} and the definition of a \(\varphi\)-switching pair, \( \varphi' \) preserves the length of sequences, i.e.\ for every \( u \in T \)
\[ 
\lh(\varphi'(u)) = \lh(u).
 \]

 \begin{claim} \label{claim:isometry}
\( \varphi' \) is still an isometry between \( S \) and \( T \).
\end{claim}

\begin{proof}[Proof of the claim]
First we show that 
\begin{enumerate}
\item[($\dagger$)]
if \( u \in T \) belongs to a \(\varphi\)-switching pair \( ( v_0, v_1 ) \) and \( v \neq  v_0, v_1  \) then
\[ 
d_T(u,v) = d_T(v_0,v) = d_T(v_1,v).
 \] 
\end{enumerate}
Indeed, if \( v_0 \not\subseteq v \) then \( d_T(u,v) > r_{\lh(v_0)}  = d_T(v_0,v_1) \geq d_T(u,v_0),d_T(u,v_1) \), whence the result follows from Fact~\ref{fct:basicsonultra}\ref{fct:basicsonultra-1}.
If instead \( v_0 \subsetneq v \), then we cannot have \( v_1 \subseteq v \) because \( v_1 \) is a terminal node of \( T \) and \( v \neq v_1 \); it follows that \( v \sqcap v_0 = v \sqcap v_1 = v \sqcap u = v_0 \), whence the desired equalities follow.

Let now \( u,v \in T \) be distinct. We distinguish various cases. If neither \( u \) nor \( v \) belong to a \(\varphi\)-switching pair, then \( \varphi'(u) = \varphi(u) \) and \( \varphi'(v) = \varphi(v) \), whence \( d_S(\varphi'(u),\varphi'(v)) = d_S(\varphi(u), \varphi(v)) = d_T(u,v) \) because \(\varphi\) is an isometry. So we can assume without loss of generality that \( u \) belongs to some \(\varphi\)-switching pair \( ( v_0, v_1 ) \). If \( v \) belongs to the same \(\varphi\)-switching pair, then \( \varphi'(u) = \varphi(v) \) and \( \varphi'(v) = \varphi(u) \), whence 
\[ 
d_S(\varphi'(u),\varphi'(v)) = d_S(\varphi(v), \varphi(u)) = d_T(v,u) = d_T(u,v).
 \] 
Thus we can further assume that \( v \neq v_0,v_1 \). 
Assume first that \( v \) does not belong to any other \(\varphi\)-switching pair and let \( i \in \{ 0,1 \} \) be such that \( u = v_i \). Then by (\( \dagger \)) we have
\[ 
d_S(\varphi'(u),\varphi'(v)) = d_S(\varphi'(v_i), \varphi'(v)) = d_S(\varphi(v_{1-i}), \varphi(v)) = d_T(v_{1-i},v) = d_T(u,v).
 \] 
Finally, assume that \( v  = w_j \) for some \(\varphi\)-switching pair  \( ( w_0,w_1 ) \) different (hence disjoint) from \( ( v_0, v_1 ) \). Then by (\( \dagger \)) again we have that
\[ 
d_T(u,v) = d_T(v_k,w_l ) \quad \text{for every } k,l \in \{ 0,1 \}.
 \] 
Thus if \( i \in \{ 0,1 \} \) is again such that \( u = v_i \) we have
\begin{multline*}
d_S(\varphi'(u),\varphi'(v)) = d_S(\varphi'(v_i), \varphi'(w_j)) = \\ = d_S(\varphi(v_{1-i}), \varphi(w_{1-j})) = d_T(v_{1-i},w_{1-j}) = d_T(u,v).  \qedhere
 \end{multline*}
\end{proof}

Finally, we claim that \( \varphi' \) is an isomorphism between the \( \alpha \)-trees \( T \) and \( S \). Clearly \( \varphi' \), being an isometry, is a bijection. Moreover, we already observed that \( \varphi' \) preserves the length of sequences. We have to check that \( \varphi' \) preserves the inclusion relation as well. Using Claim~\ref{claim:isometry} and Remark~\ref{rmk:crucial2}, for every distinct \( u,v \in T \) we have
\begin{multline*}
u \subseteq v \iff d_T(u,v) = r_{\lh(u)} \iff \\ \iff d_S(\varphi'(u),\varphi'(v)) = r_{\lh(u)} = r_{\lh(\varphi'(u))} \iff \varphi'(u) \subseteq \varphi'(v).   
 \end{multline*}
 Thus \( \varphi' \) witnesses \( T \cong S \) and we are done.
\end{proof}

We are now ready to prove Theorem~\ref{thm:maincammarmot}.

\begin{proof}[Proof of Theorem~\ref{thm:maincammarmot}]
\ref{thm:maincammarmot-1} Let \( D \in \mathcal{D}_\Sep \) be ill-founded. Then setting \( \kappa = \alpha = \omega \) and identifying \( \T{\omega}{\omega} \) with the set \( \mathrm{Tr} \) of trees on \(\Nat\), 
we get that the map \( T \mapsto X_T \) defined at the beginning of this section may be construed as a Borel map from \( \mathrm{Tr} \) to \( \US^D_\Sep \): since by Theorem~\ref{thm:red} such map 
reduces isomorphism to isometry, we get
\[ 
{\cong_{\mathsf{TREES}}} \leq_B {\cong^i_D }.
 \] 
The result then follows from Theorem~\ref{thm:upperboundultrametric}\ref{thm:upperboundultrametric-2} and \( {\cong_{\mathsf{TREES}}} \sim_B {\cong_{\mathsf{GRAPHS}}} \).

\ref{thm:maincammarmot-2} The proof of this part uses two jump operators for equivalence relations introduced by H.\ Friedman in~\cite{friedman2000} (where $E^+$ is actually denoted by $ {\rm NCS}(X,E)$) and~\cite[\S1.2.2]{Friedman1989}, respectively. Our definition of the operator \( E \mapsto E^+ \) is not the original one, but it is equivalent to it (see~\cite[Lemma 3.11]{Camerlo:2015ar}).

\begin{defin}
Given an equivalence relation \( E \) with standard Borel domain \( X \), let \( E^+ \) and \( E^\omega \) be the equivalence relations on the countable product \( \pre{\Nat}{X} \) of \( X \) defined by 
(\( [x]_E \) denotes the \( E \)-equivalence class of \( x \)):
\begin{align*}
( x_m )_{ m \in \Nat } \mathrel{E^+} ( y_m )_{ m \in \Nat} & \iff \{ [x_m]_E \mid m \in \Nat \} = \{ [y_m]_E \mid m \in \Nat \} \\
& \iff \forall m \, \exists k , l \, (x_m \mathrel{E} y_k \text{ and } y_m \mathrel{E} x_l),
 \end{align*}
 and
\[
( x_n )_{ n \in \Nat } \mathrel{E^\omega} ( y_n )_{ n \in \Nat}
\iff
\exists f \colon \Nat \to \Nat \text{ bijective } \forall n (x_n \mathrel{E} y_{f(n)}).
\]
\end{defin}
It is not hard to see that if \( E \) is Borel, then so are \( E^+ \) and \( E^\omega \). 
Furthermore, a (canonical) Borel reduction  of \( E^+ \) to \( E^\omega \) may be obtained as follows. Fix any bijection \( \langle \cdot,\cdot \rangle \colon \Nat \times \Nat \to \Nat  \). Given \( (x_n)_{n \in \Nat} \in \pre{\Nat}{X} \), set  \( x'_{\langle n,m \rangle} = x_m \), so that \( (x'_m)_{m \in \Nat} \) is an enumeration of the \( x_m \)'s in which each of them is repeated infinitely many times. It is not hard to check, that \( (x_m)_{m \in \Nat} \mapsto (x'_m)_{m \in \Nat} \) is a Borel map, and that  given \( (x_m)_{m \in \Nat}, (y_m)_{m \in \Nat} \in \pre{\Nat}{X} \)	
\begin{equation}  \label{eq:redjump}
(x_m)_{m \in \Nat} \mathrel{E^+} (y_m)_{m \in \Nat} \iff (x'_m)_{m \in \Nat} \mathrel{E^+} (y'_m)_{m \in \Nat} \iff (x'_m)_{m \in \Nat} \mathrel{E^\omega} (y'_m)_{m \in \Nat}.
 \end{equation}

The fundamental fact about the jump operator \( E \mapsto E^+ \) is the following:

\begin{fact}[H.~Friedman-Stanley, see~{\cite[Theorem 8.3.6]{gaobook}}] \label{fct:jump}
If \( E \) is a Borel equivalence relation on a standard Borel space with at least two classes, then \( E <_B E^+ \).
\end{fact}
Since we showed that \( E^+ \leq_B E^\omega \), it follows that also \( E <_B E^\omega \) whenever \( E \) is a non-trivial Borel equivalence relation with standard Borel domain.

To prove that \( {\cong^i_D} \leq_B {\cong^i_{D'}} \iff \alpha_D \leq \alpha_{D'} \)  for all well-founded \( D,D' \in \mathcal{D}_\Sep \) we will use the following two claims.

\begin{claim} \label{claim:successor}
Let \( D, D' \in \mathcal{D}_\Sep \) be well-founded and such that \( \alpha_{D'} = \alpha_D + 1 \). If  \( \cong^i_D \) is Borel, then%
\footnote{With a more involved argument, one can actually show that \( {\cong^i_{D'}} \sim_B {(\cong^i_D)^\omega} \) (see~\cite[Theorem 5.13]{Camerlo:2015ar}).}
 \(  {\cong^i_{D}} <_B {\cong^i_{D'}} \). Moreover, \( \cong^i_{D'} \) is Borel too.
\end{claim}

\begin{proof}[Proof of the claim]
By Lemma~\ref{lemma:transfer}, we may assume without loss of generality that \( D' = D \cup \{ r \} \) with \( r \in \RR_{> 0 } \) greater than all elements of \( D \), so that in particular \( \US^D_\Sep \subseteq \US^{D'}_\Sep \) and hence \( {\cong^i_{D}} \leq_B {\cong^i_{D'}} \). To prove that \( {\cong^i_{D'}} \nleq_B {\cong^i_{D}} \), by Fact~\ref{fct:jump} it is enough to show that 
\[ 
{(\cong^i_D)^+} \leq_B {\cong^i_{D'}}.
 \] 

Let
\[ 
\pre{\Nat}{(\US^D_\Sep)} \to \pre{\Nat}{(\US^D_\Sep)} \qquad (X_m)_{m \in \Nat} \mapsto (X'_m)_{m \in \Nat}
 \] 
be the canonical reduction of \( (\cong^i_D)^+ \) to \( (\cong^i_D)^\omega \) described before Fact~\ref{fct:jump}. Let
\( X = (X,d) \) be the ultrametric space defined on the disjoint union \( X \) of the \( X'_m \)'s by setting
 \[ 
d(x,y) = 
\begin{cases}
d_m(x,y) & \text{if \( x,y \) belong to the same }X'_m \text{ for some } m \in \Nat\\
r & \text{otherwise},
\end{cases}
 \] 
where \( d_m \) is the distance function of the space \( X'_m \). Notice that a pair of points \( x,y \in X \) belongs to the same component \( X'_m \) if and only if \( d(x,y) < r \). Clearly, \( X \in \US^{D'}_{\Sep} \), and it is not hard to check that the map 
\[ 
\pre{\Nat}{(\US^D_\Sep)} \to \US^{D'}_\Sep \qquad (X_m)_{m \in \Nat} \mapsto X
 \] 
is Borel: we claim that it is also a reduction of \( (\cong^i_D)^+ \) to \( \cong^i_{D'} \). Let \( (X_m)_{m \in \Nat}, (Y_m)_{m \in \Nat} \in  \pre{\Nat}{(\US^D_\Sep)} \). If \( (X_m)_{m \in \Nat} \mathrel{{(\cong^i_D)}^+} (Y_m)_{m \in \Nat} \), then \( (X'_m)_{m \in \Nat} \mathrel{{(\cong^i_D)}^\omega} (Y'_m)_{m \in \Nat} \), i.e.\ there are a bijection \( f \colon \Nat \to \Nat \) and isometries \( \varphi_m \colon X'_m \to Y'_{f(m)} \); it is then straightforward to check that then \( \varphi = \bigcup_{m \in \Nat} \varphi_m \) is an isometry between \( X \) and \( Y \). Conversely, if \( \varphi \) is an isometry, then \( \varphi \) maps each of the \( X'_m \)'s onto some \( Y'_n \) by the way we defined the metrics on \( X \) and \( Y \) (in particular, by the fact that the distance \( r \) is used only to link different components of the disjoint union): this implies that there is a bijection \( f \colon \Nat \to \Nat \) such that \( \varphi \restriction X'_m \) is an isometry between \( X'_m \) and \( Y'_{f(m)} \), thus \( (X'_m)_{m \in \Nat} \mathrel{{(\cong^i_D)}^\omega} (Y'_m)_{m \in \Nat} \), whence \( (X_m)_{m \in \Nat} \mathrel{{(\cong^i_D)}^+} (Y_m)_{m \in \Nat} \) by~\eqref{eq:redjump}.

To prove the last part of the claim, it is enough to prove that 
\[ 
{\cong^i_{D'}} \leq_B {(\cong^i_D)^\omega}.
\]
Given \( X = (X,d) \in \US^{D'}_\Sep \), let \( (X_m)_{m \in \Nat} \) be an enumeration without repetitions of all the open balls of \( X \) with radius \( r \), with the convention that \( X_m = \emptyset \) if there are 
only \( \leq m \)-many such balls (the fact that there are at most countably many such balls follows from separability of \( X \) and Fact~\ref{fct:basicsonultra}\ref{fct:basicsonultra-3}). By 
parts~\ref{fct:basicsonultra-2} and~\ref{fct:basicsonultra-3} of Fact~\ref{fct:basicsonultra}, each open ball \( X_m \neq \emptyset \), having radius \( r  \) and being clopen in \( X \), 
is an element of \( \US^D_\Sep \). Thus using the Borel selectors \( \psi_n \) and arguing as in Claim~\ref{claim:HB} one can see that the map
\[ 
\US^{D'}_\Sep \to \pre{\Nat}{(\US^D_\Sep)} \qquad X \mapsto (X_m)_{m \in \Nat}
 \] 
is Borel. We now check that it also reduces \( \cong^i_{D'} \) to \( (\cong^i_D)^\omega \). On the one hand, any isometry \( \varphi \colon X \to Y \) clearly induces a bijection \( f \colon \Nat \to \Nat \) such that \( X_m = \emptyset \iff Y_{f(m)} = \emptyset \), and \( \varphi(X_m) = Y_{f(m)} \) whenever \( X_m \neq \emptyset \); thus \( f \) witnesses \( (X_m)_{m \in \Nat} \mathrel{(\cong^i_D)^\omega} (Y_m)_{m \in \Nat} \). 
On the other hand, if \( f \colon \Nat \to \Nat \) is a bijection such 
 that \( X_m \cong^i Y_{f(m)} \) via some isometry \( \varphi_m \), then \( \varphi = \bigcup_{m \in \Nat} \varphi_m \) is an isometry between \( X \) and \( Y \): indeed  by Fact~\ref{fct:basicsonultra}\ref{fct:basicsonultra-2} the map \( \varphi \) is well-defined because the \( X_m \)'s form a partition of \( X \), and any two points of \( X \) (respectively, \( Y \)) 
 at distance \( < r \) are in the same open ball of radius \( r \), whence points in distinct open balls with radius \( r \) have distance exactly \( r \) because this is the unique remaining distance in \( D' \).
 \end{proof}

\begin{claim} \label{claim:limit}
Let \( D \in \mathcal{D}_\Sep  \) be well-founded with \( \alpha_D \) limit, and assume that \( \cong^i_{D'} \) is Borel for every
 well-founded \( D' \in \mathcal{D}_\Sep \) such that \( \alpha_{D'} < \alpha_D \). 
Then \( \cong^i_D \) is Borel as well.
\end{claim}

\begin{proof}[Proof of the claim]
Let \( (r_\beta)_{\beta < \alpha_D} \) be the increasing enumeration of \( D \), and for each \( \beta \leq \alpha_D \) set \( D_\beta = \{ r_\gamma \mid \gamma < \beta \} \).
Given \( X = (X,d) \in \US^D_\Sep \), \( x \in X \), and \( \beta < \alpha_D \), set 
\[ 
C^X_\beta(x) = B_{r_{\beta+1}}(x) \setminus B_{r_\beta}(x) =  \{ y\in X \mid d(x,y) = r_\beta \} . 
\]
Notice that by Fact~\ref{fct:basicsonultra}\ref{fct:basicsonultra-3} the set \( C^X_\beta(x) \) is  clopen, and by Fact~\ref{fct:basicsonultra}\ref{fct:basicsonultra-1} it 
belongs to 
\( \US^{D_{\beta+1}}_\Sep \). Moreover,  using an argument analogous to that of Claim~\ref{claim:HB} one sees that for every \( \beta < \alpha_D \) and \( n \in \Nat \) the map
\[ 
\US^D_\Sep \to  \US^{D_{\beta+1}}_\Sep  \qquad X \mapsto C^X_\beta(\psi_n(X)),
 \] 
 is Borel. Thus given \( X = (X,d_X) \) and \( Y = (Y,d_Y) \) in \( \US^D_\Sep \) the condition
\begin{equation}\tag{$\dagger$} \label{cond:borel}
\exists n \in \Nat \, \forall \beta < \alpha_D \, (C^{X}_{\beta}(\psi_0(X)) \isom^i_{D_{\beta+1}} C^{Y}_{\beta}(\psi_n(Y)))
\end{equation}
is Borel (here we use the fact that by hypothesis each \( \cong^i_{D_{\beta+1}} \) is Borel): 
we claim that \( X \cong^i Y \) if and only if condition~\eqref{cond:borel} holds, so that \( \cong^i_D \) is Borel.

Let \( \varphi \colon X \to Y \) be an isometry. The space \( Y \) is discrete (since \( D \) is bounded away from \( 0 \)), hence there is \( n \in \Nat \) such that \( \varphi(\psi_0(X)) = \psi_n(Y) \). It follows that 
\( \varphi \restriction C^{X}_{\beta}(\psi_0(X)) \) witnesses \( C^{X}_{\beta}(\psi_0(X)) \isom^i_{D_{\beta+1}} C^{Y}_{\beta}(\psi_n(Y)) \) (for every \( \beta < \alpha_D \)), hence~\eqref{cond:borel} holds. 
Conversely, assume that~\eqref{cond:borel} holds, let \( n \in \Nat \) be a witness of this, and let \( \varphi_\beta \colon C^{X}_{\beta}(\psi_0(X)) \to C^{Y}_{\beta}(\psi_n(Y)) \) be an isometry: we claim that the 
map \( \varphi \colon X \to Y \) defined by
\[ 
\varphi(x) = 
\begin{cases}
\psi_n(Y) & \text{if }x = \psi_0(X) \\
\varphi_\beta(x) & \text{if } x \in C^{X}_{\beta}(\psi_0(X))
\end{cases}
 \] 
is an isometry. To see that \( \varphi \) is a bijection notice that the families \( \{ C^{X}_{\beta}(\psi_0(X)) \mid \beta < \alpha_D \} \) and \( \{ C^{Y}_{\beta}(\psi_n(Y)) \mid \beta < \alpha_D \} \) 
are partitions of \( X \setminus \{ \psi_0(X) \} \) and \( Y \setminus \{ \psi_n(Y) \} \), respectively, so that the result follows from the fact that the maps 
\( \varphi_\beta \colon C^{X}_{\beta}(\psi_0(X)) \to C^{Y}_{\beta}(\psi_n(Y)) \) are bijective.  
Hence it remains to show that \( \varphi \) is distance preserving.
Let \( x,x' \in X \). If \( x= x' \) or \( x,x' \in C^{X}_{\beta}(\psi_0(X)) \) for the same \( \beta < \alpha_D \), this is trivial. If one of \( x,x' \) equals \( \psi_0(X) \) and the other one belongs to \(  C^{X}_{\beta}(\psi_0(X)) \), 
then \( d_X(x,x') = r_\beta \) by definition of \( C^{X}_{\beta}(\psi_0(X)) \): since in this case one of \( \varphi(x), \varphi(x') \) equals \( \psi_n(Y) \) and the other one belongs to 
\( C^{Y}_{\beta}(\psi_n(Y)) \) by definition of \( \varphi \), it follows that \( d_Y (\varphi(x), \varphi(x')) = r_\beta \) as well, so that \( d_X(x,x') = d_Y (\varphi(x), \varphi(x')) \). Finally, suppose that 
\( x \in C^{X}_{\beta}(\psi_0(X)) \) and \( x' \in C^{X}_{\beta'}(\psi_0(X)) \) for distinct \( \beta, \beta' < \alpha_D \), and assume without loss of generality that \( \beta < \beta' \). Then 
\( d_X(x', \psi_0(X)) = r_{\beta'} > r_\beta = d_X(x,\psi_0(X)) \), whence \( d_X(x,x') = r_{\beta'} \) by Fact~\ref{fct:basicsonultra}\ref{fct:basicsonultra-1}. By definition of \( \varphi \), we have \( \varphi(x) \in  C^{Y}_{\beta}(\psi_n(Y)) \) and 
\( \varphi(x') \in  C^{Y}_{\beta'}(\psi_n(Y)) \): arguing as above we then get that \( d_Y(\varphi(x), \varphi(x')) = r_{\beta'} \) as well, whence \(  d_X(x,x') = d_Y (\varphi(x), \varphi(x')) \) again.
\end{proof}

By Lemma~\ref{lemma:transfer}, if \( D,D' \in \mathcal{D}_\Sep \) are well-founded and \( \alpha_D \leq \alpha_{D'} \), then \( {\cong^i_D} \leq_B {\cong^i_{D'}} \). Moreover, using 
Claims~\ref{claim:successor} and~\ref{claim:limit} one can easily prove by induction on \( \alpha_D \) that if \( D \in \mathcal{D}_\Sep \) is well-founded, then the equivalence relation \( \cong^i_D \) is Borel and that its 
complexity with respect to \( \leq_B \) 
(strictly) increases with \( \alpha_D \). Thus the isometry relations \( \cong^i_D \) for \( D \in \mathcal{D}_\Sep \) well-founded form an \(\omega_1 \)-chain of Borel equivalence relations. To prove that they are also \( \leq_B \)-cofinal among Borel 
equivalence relations classifiable by countable structures, we use the following fact, again due to H.\ Friedman. Recall that to any well-founded tree \( T \) on \(\Nat\) we can associate a rank function \( \rho_T \colon T \to \omega_1 \) by setting for \( t \in T \)
\[ 
\rho_T(t) = 
\begin{cases}
0 & \text{if \( t \) is terminal in } T \\
\sup \{ \rho_T(s)+1 \mid {s\in T} \text{ and } {t \subsetneq s} \} & \text{otherwise}.
\end{cases}
 \] 
The ordinal \( \rho_T(\emptyset) < \omega_1 \) is called \emph{rank of \( T \)}, and the set \( \mathcal{T}_\alpha \subseteq \mathrm{WF} \) of well-founded trees of rank \( < \alpha \) (for \( 1 \leq \alpha < \omega_1 \)) is a Borel subset of the Polish space \( \mathrm{Tr} \) (see e.g.~\cite[Section 2.E and Sections 34--35]{Kechris1995}). 

\begin{fact}[H.~Friedman, see~{\cite[Theorem 13.2.15]{gaobook}}] \label{fct:cofinal}
Let \( \cong_\alpha \) be the equivalence relation of isomorphism on \( \mathcal{T}_\alpha \). Then the sequence \( (\cong_\alpha)_{1 \leq \alpha < \omega_1} \) is \( \leq_B \) strictly increasing and 
\( \leq_B \)-cofinal among Borel equivalence relations
 classifiable by countable structures.
\end{fact}

In view of Fact~\ref{fct:cofinal}, it is thus enough to show that for all nonzero \( \alpha < \omega_1 \) there is a well-founded \( D \in \mathcal{D}_\Sep \) such that \( \alpha_D = \alpha+1 \) and 
\[ 
{\cong_\alpha} \leq_B {\cong^i_D}.
 \] 
We argue by induction on \( \alpha \). The basic case \( \alpha = 1 \) is trivial, as there is only one tree \( T \) on \( \Nat \) of rank \( 0 \), namely \( T = \{ \emptyset \} \). For the inductive step, let \( \alpha > 1 \) 
and let \( D_\beta \in \mathcal{D}_\Sep \) for \( 1\leq \beta < \alpha \) be such that \( \alpha_{D_\beta} = \beta+1 \) and \( f_\beta \colon {\cong_\beta} \leq_B {\cong^i_{D_\beta}} \). 
By Lemma~\ref{lemma:transfer} we 
can assume without loss of generality that \( D_\beta \subseteq [0;1] \) and \( D_\beta \subseteq D_{\beta'} \) for all \( \beta \leq \beta' < \alpha \), so that 
\( D' = \bigcup_{1 \leq \beta < \alpha} D_\beta \in \mathcal{D}_\Sep \) is a well-founded subset of the unit interval with \( \alpha_{D'} = \alpha \). Let \( D = D' \cup \{ 2 \} \), so that \( \alpha_D = \alpha+1 \). 
Given \( T \in \mathcal{T}_\alpha \) and \( i \in \Nat \), let 
\[ 
T_i = \{ (s_0, \dotsc, s_{n-1}) \in \pre{<\Nat}{\Nat} \mid ( i , s_0, \dotsc, s_{n-1}) \in T \} .
\] 
Let also \( \beta_i < \alpha \) be the rank of \( T_i \). Then let 
\( Y_T = (Y_T,d_T)  \in \US^D_\Sep \) be the space defined on the disjoint union \( Y_T \) of the spaces \( f_{\beta_i}(T_i) \in \US^{D_{\beta_i}}_\Sep \subseteq \US^{D'}_\Sep \) (with \( i \in \Nat \)) by setting 
\[ 
d_T(x,y) = 
\begin{cases}
d_i(x,y) & \text{if \( x,y \) belong to the same }  f_{\beta_i}(T_i) \text{ for some } i \in \Nat \\
2 & \text{otherwise,}
\end{cases}
 \] 
where \( d_i \) is the distance function of \(  f_{\beta_i}(T_i) \). It is not hard to see that \( d_T \) is a complete ultrametric, whence \( Y_T \in \US^D_\Sep \). Moreover, the map 
\[
\mathcal{T}_\alpha \to \US^D_\Sep \qquad T \mapsto Y_T 
\]
is clearly Borel and reduces \( \cong_\alpha \) to \( \cong^i_D \), as required.

This concludes the proof of Theorem~\ref{thm:maincammarmot}.
\end{proof}

\section{The complexity of the isometry relation on non-separable spaces} \label{sec:nonseparable}

In recent years, various extensions of descriptive set theory have been proposed with the goal of being able to handle more and more topological spaces that may be useful for applications in other areas of 
mathematics and computer science. One of these possibilities is suggested by the following natural problem. A large part of (classical) descriptive set theory is concerned with the Cantor space \( \pre{\Nat}{2} \) and the Baire 
space \( \pre{\Nat}{\Nat} \):
\begin{quote}
to what extent the classical results remain valid (\emph{mutatis mutandis}) when replacing \( \Nat = \omega \) with an uncountable cardinal \( \kappa \)? 
\end{quote}
The resulting theory is sometimes called \emph{generalized descriptive set theory}, and it is mainly concerned with the following basic spaces. 

\begin{defin} \label{def:generalizedBaire}
Let \( \kappa \) be an infinite cardinal. The \emph{generalized (\( \kappa \)-)Baire space} is the space 
\[ 
\pre{\kappa}{\kappa} = \{ f \mid f \colon \kappa \to \kappa \} 
\] 
endowed with the \emph{bounded topology} \( \tau_b \), i.e.\ with the topology generated by the basic open sets
\[ 
\Nbhd^\kappa_s = \{ f \in \pre{\kappa}{\kappa} \mid s \subseteq f \}
 \] 
for \( s \in \pre{< \kappa}{\kappa} \). The \emph{generalized (\( \kappa \)-)Cantor space} is the closed subspace of \( \pre{\kappa}{\kappa} \) given by
\[ 
\pre{\kappa}{2} = \{ f \in \pre{\kappa}{\kappa} \mid f \colon \kappa \to \{ 0,1 \} \}.
 \] 
\end{defin}

\begin{remark}
\begin{enumerate-(a)}
\item
By taking \( \kappa = \omega \) in Definition~\ref{def:generalizedBaire} we obtain the usual Baire and Cantor space.
\item
The bounded topology \( \tau_b \) is not the unique choice for a natural topology on \( \pre{\kappa}{\kappa} \): for example, one could also consider the \emph{product topology} \( \tau_p \), 
i.e.\ one could see \( \pre{\kappa}{\kappa} \) as the product of \( \kappa \)-many copies of \( \kappa \) and endow it with the product of the discrete topology on \( \kappa \). Although when \( \kappa = \omega \) the 
bounded and the product topologies coincide, when \( \kappa \) is uncountable the topology \( \tau_b \) becomes strictly finer than \( \tau_p \) --- see~\cite{AM} for more on this.
\end{enumerate-(a)}
\end{remark}

When studying the generalized Baire and Cantor spaces, it is natural to generalize accordingly all the corresponding topological notions: this is essentially obtained by systematically replacing in all definitions \( \omega \) with \( \kappa \)
where appropriate.
Recall that a 
\emph{\( \kappa^+ \)-algebra on a set \( X \)} is a nonempty collection of subsets of \( X \) which contains \( \emptyset \) and is closed under complements and unions of size at most \( \kappa \).

\begin{defin} \label{def:kappaBorel}
Let \( \kappa \) be an infinite cardinal.
\begin{enumerate-(1)}
\item
A subset of \( \pre{\kappa}{\kappa} \) is called \emph{\( \kappa^+ \)-Borel} if it belongs to the smallest \( \kappa^+ \)-algebra on \( \pre{\kappa}{\kappa} \) containing all the \( \tau_b \)-open sets.
\item
A function \( f \colon \pre{\kappa}{\kappa} \to \pre{\kappa}{\kappa} \) is \emph{\( \kappa^+ \)-Borel} if the preimage of every open set (equivalently, of every \( \kappa^+ \)-Borel set) is \( \kappa^+ \)-Borel.
\end{enumerate-(1)}
\end{defin}

\begin{defin} \label{def:kappaanalytic}
Let \( \kappa \) be an infinite cardinal.
A subset of \( \pre{\kappa}{\kappa} \) is called \emph{\( \kappa \)-analytic} if it is a continuous image of a closed subset of \( \pre{\kappa}{\kappa} \).
\end{defin}

\begin{remark} \label{rmk:generalizedregular}
\begin{enumerate-(a)}
\item
When \( \kappa = \omega \), Definition~\ref{def:kappaanalytic} is equivalent to the usual definition of an analytic set, and could be equivalently reformulated also as: a set is (\(\omega\)-)analytic if it is a continuous image of 
the whole \( \pre{\Nat}{\Nat} \). In contrast, when \( \kappa \) is uncountable the latter reformulation is no more equivalent to that of Definition~\ref{def:kappaanalytic}, as shown in~\cite{LS}.
\item \label{rmk:generalizedregular-b}
It can be shown that, as in the classical case \( \kappa = \omega \), the \( \kappa^+ \)-Borel sets form a proper non-collapsing hierarchy of subsets of \( \pre{\kappa}{\kappa} \) of length \( \kappa^+ \) 
(but this requires brand new arguments if%
\footnote{By definition, \( \kappa^{< \kappa} = \sup \{ \kappa^\lambda \mid \lambda \text{ is a cardinal } < \kappa \} \). Equivalently, \( \kappa^{< \kappa} \) is the cardinality of the set of sequences \( \kappa^{< \alpha} \) introduced at the beginning of Section~\ref{sec:treesisometry} when setting \( \alpha = \kappa \); this is why we can unambiguously confuse the two notations.}
 \( \kappa^{< \kappa} \neq \kappa \), see~\cite{AM}). Moreover, the collection of \( \kappa \)-analytic subsets of \( \pre{\kappa}{\kappa} \) contains all open and closed sets and is closed under unions and intersections of size (at most) \( \kappa \); it follows that, in particular, all \( \kappa^+ \)-Borel sets are \( \kappa \)-analytic. The converse does not hold in general: indeed, when \( \kappa \) is regular and uncountable it is even no more true that \( \kappa^+ \)-Borel sets coincide with the collection of \( A \subseteq \pre{\kappa}{\kappa}\) such that both \( A \) and \( \pre{\kappa}{\kappa} \setminus A \) are \( \kappa \)-analytic, see e.g.\ \cite{Friedman:2011nx}.
\end{enumerate-(a)}
\end{remark}

Definitions~\ref{def:kappaBorel} and~\ref{def:kappaanalytic} can straightforwardly be adapted to 
subspaces of \( \pre{\kappa}{\kappa} \), notably including the generalized Cantor space \( \pre{\kappa}{2} \). Moreover, since the 
product of finitely many copies of \( \pre{\kappa}{\kappa} \) (endowed with the product of the bounded topology) is homeomorphic to \( \pre{\kappa}{\kappa} \), such definitions can be extended to those spaces as well. 

Definition~\ref{def:kappaBorel} further suggests to consider the following generalization of the notion of a (standard) Borel space, which corresponds to the case \( \kappa = \omega \). 
(see~\cite{Motto-Ros:2011qc,AM}).

\begin{defin}
A space \( X = (X, \mathcal{B}) \) with \( \mathcal{B} \) a \( \kappa^+ \)-algebra on \( X \) is called \emph{Borel \( \kappa \)-space} if there is \( A \subseteq \pre{\kappa}{\kappa} \) and a \emph{\( \kappa^+ \)-Borel isomorphism} between \( X \) and \( A \), that is a bijection \( f \colon X \to A \) such that for every \( B \subseteq X \)
\[ 
B \in \mathcal{B} \iff f(B) \text{ is (relatively) \( \kappa^+ \)-Borel in } A .
 \] 
 
If \( A \) can be taken to be \( \kappa^+ \)-Borel in \( \pre{\kappa}{\kappa} \), then \( X \) is called \emph{standard Borel \( \kappa \)-space}.
\end{defin}

Generalized descriptive set theory experienced a quick development in the last ten years, and the results obtained progressively revealed its deep connections
with other parts of set theory (infinite combinatorics, forcing theory, large cardinal assumptions, forcing axioms, and so on) and model theory. One striking example of this phenomenon is the study of the isomorphism 
relation between uncountable structures. 

By replacing \( \omega \) with an uncountable \( \kappa \) in the setup of page~\pageref{pageref:models} 
(see the paragraph after Definition~\ref{def:classifiable} and Remark~\ref{rmk:models}), one obtains that the space of (codes for) \( \mathcal{L} \)-structure of size \( \kappa \)
\[ 
\Mod^\kappa_{\mathcal{L}} = \prod_{i < I} \pre{\left( \pre{n_i}{\kappa} \right)}{2}.
 \] 
is homeomorphic to generalized Cantor space \( \pre{\kappa}{2} \), and that the isomorphism relation \(\cong \) on it is a \( \kappa \)-analytic equivalence relation. Further considering models of a given first-order theory \( T \) (or, more generally, of a sentence \( \varphi \) in the infinitary logic \( \L_{\kappa^+ \kappa} \)), one get the Borel \( \kappa \)-space
\[
\Mod^\kappa_\upvarphi = \{ x \in \Mod^\kappa_\L \mid M_x \models \upvarphi \}.
\]
By a straightforward generalization of an argument of Vaught (see~\cite{Vaught:1974kl} for \( \kappa = \omega_1 \) and~\cite{Friedman:2011nx,AM} for arbitrary uncountable \( \kappa \)'s), 
one can also see that if \( \kappa^{< \kappa} = \kappa \), then for every \( X \subseteq \Mod^\kappa_{\mathcal{L}} \) the following are equivalent:
\begin{itemizenew}
\item
\( X\)  is \( \kappa^+ \)-Borel and invariant under isomorphism;
\item
\( X = \Mod^\kappa_\upvarphi \) for some \( \L_{\kappa^+ \kappa} \)-sentence \( \varphi \).
\end{itemizenew}
It thus follows that if \( \kappa^{< \kappa} = \kappa \), then each \( \Mod^\kappa_\varphi \) as above is a standard Borel \( \kappa \)-space.

The theory of Borel reducibility can be straightforwardly generalized to the uncountable context as well.

\begin{defin} \label{def:reducibilitygeneralized}
Let \( E \) and \( F \) be equivalence relations defined on Borel \( \kappa \)-spaces \( X \) and \( Y \), respectively. We say that \( E \) is \emph{\( \kappa^+ \)-Borel reducible} to \( F \) (in symbols, \( E \leq^\kappa_B F \)) if there is a \( \kappa^+ \)-Borel function \( f \colon X \to Y \) such that for every \( x,y \in X \)
\[ 
x \mathrel{E} y \iff f(x) \mathrel{F} f(y).
 \] 
As in the classical case, the symbol \( <^\kappa_B \) denotes the strict part of \( \leq^\kappa_B \), while \( \sim^\kappa_B \) denotes the induced bi-reducibility.
\end{defin}

The notion of \( \kappa^+ \)-Borel reducibility can be used as a classification tool when the objects to be classified are uncountable/non-separable. This is clear when e.g.\ we want to compare the complexity of the 
isomorphism relation on the models of two \( \L_{\kappa^+ \kappa} \)-sentences, as such relations are \( \kappa \)-analytic equivalence relations on (standard, if  \( \kappa^{< \kappa} = \kappa \)) Borel \( \kappa \)-spaces. 
For example, by straightforwardly adapting the arguments used in 
the countable case one can show that for any infinite cardinal \( \kappa \) each of the following is, up to \( \kappa^+ \)-Borel bi-reducibility, the most complex isomorphism relation:
\begin{enumerate-(1)}
\item
the relation \( \cong^\kappa_{\mathsf{GRAPHS}} \) of isomorphism between graphs of size \( \kappa \);
\item
assuming  \( \kappa^{< \kappa} = \kappa \), the relation \( \cong^\kappa_{\mathsf{LO}} \) of isomorphism between linear orders of size \( \kappa \);
\item \label{kappatrees}
assuming \( \kappa^{< \kappa} = \kappa \) again, the relation \( \cong^\kappa_{\mathsf{TREES}} \) of isomorphism between trees in \( \T{\kappa}{\kappa} \).
\end{enumerate-(1)}

Along the same lines, the impressive work of S.~Friedman, 
Hyttinen, Kulikov and others (see~\cite{Friedman:2011nx,Hyttinen:2012fj,FHKFundamenta,Moreno}) revealed that, unlike in the countable case, there is a tight connection between the classification under \( \kappa^+ \)-Borel reducibility of the isomorphism relation on the models of size \( \kappa \) 
of a given first-order theory \( T \) (for suitable \emph{uncountable regular} cardinals \( \kappa \)), and the complexity of \( T \) in terms of Shelah's stability theory --- this is currently considered as one 
of the strongest 
motivations for pursuing the study of generalized descriptive set theory. Moreover, D\v{z}amonja and V\"a\"an\"anen showed that even when considering \emph{singular} cardinals \( \kappa \),  one can develop a reasonable descriptive set theory for the generalized Baire space \( \kappa^\kappa \). When \( \kappa \) is a  strong limit cardinal of countable cofinality, this theory turns out to be much closer to what one obtains in the classical context \( \kappa = \omega \) (as opposed to the case of an uncountable regular \( \kappa \), where even basic facts like the one considered in Remark~\ref{rmk:generalizedregular}\ref{rmk:generalizedregular-b}  may fail). For example, one can establish an interesting
connection between the descriptive set theoretic complexity of the chain isomorphism
orbit of a model, a natural reduction order on  trees of size \( \kappa \) with no \( \kappa \)-branches, and winning strategies in the corresponding
dynamic Ehrenfeucht-Fra\"iss\'e games (this yields a neat analog of the notion of Scott watershed from the Scott analysis of countable models, see~\cite{DzaVaa}). Further results concerning the bi-embeddability relation between uncountable structures (in both the regular and the singular case) may be found in~\cite{Motto-Ros:2011qc,AM}.

Nevertheless, the (classical) theory of Borel reducibility can be applied to a quite large variety of situations, and its scope as a classification tool is much wider than just isomorphism or bi-embeddability relations between (countable) structures. 
It is therefore natural to ask whether the same is true in the 
uncountable context, i.e.\ whether the preorder \( \leq^\kappa_B \) can be used to handle more classification problems: this would further reinforce the idea that generalized descriptive set theory is not just interesting 
because of its deep connections with other (apparently unrelated) parts of logic, but it may also be useful to understand and solve problems arising in other areas of mathematics. 

In the rest of this section we are going to show that
this is indeed the case by illustrating how \( \kappa^+ \)-Borel reducibility can be used to determine the complexity of the classification problem (up to isometry) for non-separable
complete metric spaces. 

First of all, given an uncountable cardinal \( \kappa \) we have to fix a (standard) \( \kappa^+ \)-Borel space of codings for  complete metric spaces 
of density character \( \kappa \). For technical reasons,%
\footnote{By~\cite{Katetov:1986tb}, an analogue \( \mathbb{U}_\kappa \) of the Urysohn space for complete metric spaces of density character \( \kappa \) exists if and only if \( \kappa^{< \kappa} = \kappa \). And even when such a 
\( \mathbb{U}_\kappa \) exists, the space \( F(\mathbb{U}_\kappa) \) of its closed subspaces endowed with the Effros \( \kappa^+ \)-Borel structure 
(i.e.\ the smallest \( \kappa^+ \)-algebra on \( F(\mathbb{U}_\kappa) \) generated by the 
sets as in~\eqref{eq:EffrosBorel}) would be a \(\kappa^+ \)-Borel space, but it is not clear whether it is standard or not. For these reasons, a generalization of Coding~\ref{coding1} to our new setup seems not the right move.}
we will use a generalization of Coding~\ref{coding2}; 	however, in order to better fit the setup of generalized descriptive set theory, we slightly modify it as follows.

Since any complete metric space \( X =  ( X , d ) \) of density character \( \kappa \) is completely determined by any dense subset \( D = \{ x_\alpha \mid \alpha < \kappa \} \) of it together with the distances between the 
\( x_\alpha \)'s, we can code \( X \) with the unique element \( c_X \in  \pre{\kappa \times \kappa \times \QQ_{> 0}}{2} \) such that for all \( \alpha, \beta < \kappa \) and \( q \in \QQ_{>0} \)
\begin{equation} \label{eq:defx_M}
c_X(\alpha, \beta, q) = 1 \iff d(x_\alpha,x_\beta) < q.
\end{equation}
This suggests to consider the space of codes \(\MS_\kappa   \) consisting of those \( c \in  \pre{\kappa \times \kappa \times \QQ_{> 0}}{2}  \)  satisfying the following conditions:
\begin{align*}
\forall \alpha,\beta < \kappa\, \forall q , q ' \in \QQ_{>0} \, & \left [ \text{if }q \leq q' \text{ then } c ( \alpha , \beta , q ) \leq c ( \alpha , \beta , q' ) \right ] 
\\
\forall \alpha , \beta < \kappa \, \exists q \in \QQ_{>0} \, & \left [ c ( \alpha , \beta , q ) = 1\right ] 
\\
\forall \alpha < \kappa \, \forall q \in \QQ_{>0} \, & \left[ c ( \alpha , \alpha , q ) = 1 \right] 
\\
\forall \alpha < \beta < \kappa \, \exists q \in \QQ_{>0} \, & \left [ c ( \alpha , \beta , q ) = 0 \right ] 
\\
\forall \alpha , \beta < \kappa \, \forall q \in \QQ_{>0} \, & \left [ c ( \alpha , \beta , q ) = 1 \iff c ( \beta , \alpha , q ) = 1 \right] 
\\
\forall \alpha , \beta , \gamma < \kappa \, \forall q,q' \in \QQ_{>0} \, & \left[ \text{if } c(\alpha,\beta,q) = 1 \text{ and } c ( \beta , \gamma , q ' ) = 1 \text{ then } c ( \alpha , \gamma , q + q ' ) = 1 \right ] 
\\
\forall \alpha<\kappa \, \exists \beta < \kappa \, \exists  q \in \QQ_{>0} \forall \gamma < \alpha \, & [ c ( \gamma , \beta , q ) = 0 ].
\end{align*}
The first six conditions are designed so that given any \( c \in \MS_\kappa \), the (well-defined) map \( d_c \colon \kappa \times \kappa \to \RR \) defined by setting 
\[ 
d_c(\alpha,\beta)  =  \inf \{q \in \QQ_{>0} \mid c(\alpha,\beta,q) = 1 \} 
 \] 
is a metric on \( \kappa \); denote by \( M_c \) the completion of \( ( \kappa , d_c ) \), and notice that the last condition ensures that \( M_c \) has density character \( \kappa \).%
\footnote{This is why we do not need to form any quotient of \( (\kappa, d_c ) \) before completing it to the space \( M_c \) (compare this with the construction \( c \mapsto M_c \) in Coding~\ref{coding2}): the fact that we now 
want to deal only with spaces of density character
\emph{exactly} \( \kappa \) enables us to always require that distinct elements of \( \kappa \) have positive distance. In the countable case, instead, this was not possible because there we wanted to include also 
codes for finite spaces.} 
It is straightforward to check that the code \( c_X \) from~\eqref{eq:defx_M} of any complete metric space \( X \) of density character \( \kappa \) belongs to \( \MS_\kappa \) and is such that \( X \) is isometric to 
\( M_{c_X} \); conversely, for each \( c \in \MS_\kappa \) the space \( M_c \) is a complete metric space of density character \( \kappa \). Therefore we can view \( \MS_\kappa \) as the space of codes for all complete 
metric spaces of density character \( \kappa \).

It is clear that when \( \kappa = \alzero \) we get a coding \( \MS_\alzero \) which is obviously equivalent to the one from Coding~\ref{coding2}. The reason for introducing this further variant is that the space 
\( \pre{\kappa \times \kappa \times \QQ_{> 0}}{2} \) can now be canonically embedded into \( \pre{\kappa \times \kappa \times \kappa}{2} \) (using any bijection between \( \QQ_{> 0} \) and \( \alzero \leq \kappa \)), and the 
latter space is in turn canonically isomorphic to the generalized Cantor space 
\( \pre{\kappa}{2} \); therefore the space \( \pre{\kappa \times \kappa \times \QQ_{> 0}}{2} \) can be identified with a closed subspace of \( \pre{\kappa}{2} \), and it is straightforward to check that 
under such identification our space of codings \( \MS_\kappa \) becomes a 
\( \kappa^+ \)-Borel subspace of \( \pre{\kappa}{2} \).
It thus follows that the relation of isometry \( \cong^i \) is a \( \kappa \)-analytic equivalence relation on the standard Borel \( \kappa \)-space \( \MS_\kappa \), and thus its restriction to any natural subclass of 
\( \MS_\kappa \) can be analyzed in terms of \( \kappa^+ \)-Borel reducibility.

Some results we obtained in the separable case straightforwardly generalize to the non-separable context as well (with essentially the same proof, although we now use a slightly  different coding). 
For example, we have the following classification result (compare it with Theorem~\ref{thm:upperboundultrametric}).

\begin{theorem}\label{thm:upperboundultrametrickappa}
\begin{enumerate-(1)}
\item \label{thm:upperboundultrametrickappa-1}
The isometry relation on all complete metric spaces of density character \( \kappa \) and \emph{size \( \kappa \)} (including thus the \emph{discrete} ones) is \( \kappa^+ \)-Borel reducible to 
\( \cong^\kappa_{\mathsf{GRAPHS}} \).
\item \label{thm:upperboundultrametrickappa-2}
The isometry relation on all complete \emph{ultrametric} spaces of density character \( \kappa \)  is \( \kappa^+ \)-Borel reducible to \( \cong^\kappa_{\mathsf{GRAPHS}} \).
\end{enumerate-(1)}
\end{theorem}

It can also be shown that the classification provided in Theorem~\ref{thm:upperboundultrametrickappa}\ref{thm:upperboundultrametrickappa-1} is optimal.

\begin{theorem} \label{thm:graphskappaanddiscrete}
The isomorphism relation \( \cong^\kappa_{\mathsf{GRAPHS}} \) is \( \kappa^+ \)-Borel reducible to the isometry 
relation on discrete (hence also locally compact) complete metric spaces of density character \( \kappa \).
\end{theorem}

\begin{proof}
Given a graph \( G \) on \( \kappa \), let \( X_G \) be the discrete complete metric space on \( \kappa \) whose distance \( d_G \) is defined by setting \( d_G(\alpha,\beta) = 1 \) if \( \alpha \) and \( \beta \) are 
adjacent in \( G \), and \( d_G(\alpha,\beta) = 2 \) otherwise (for all distinct \( \alpha,\beta \in \kappa \)). It is straightforward to check that the map 
\( G \mapsto X_G \) is a \( \kappa^+ \)-Borel reduction of \( \cong^\kappa_{\mathsf{GRAPHS}} \) to \( \cong^i \) whose range consists discrete spaces.
\end{proof}

It is instead not clear in general whether the classification from  Theorem~\ref{thm:upperboundultrametrickappa}\ref{thm:upperboundultrametrickappa-2} is optimal as well. However, we are now going to present a very 
strong anti-classification result showing that \emph{consistently} the isometry relation on \emph{discrete} complete \emph{ultrametric} spaces of density character \( \kappa \) (for almost any uncountable successor cardinal 
\( \kappa \)) is 
as complex as possible with respect to \( \leq^\kappa_B \). This fact may be interpreted as saying that it is not possible to find in \( \mathsf{ZFC} \) alone a satisfactory classification for non-separable complete metric spaces,
even when restricting the attention to extremely simple spaces like the discrete (hence also locally compact) ultrametric ones.

\begin{theorem}[\( \mathsf{ZF} \)] \label{thm:mainkappa} 
Assume \( \mathsf{V=L} \). Let \( \kappa = \lambda^+ \) be such that \( \lambda^\omega = \lambda \) (equivalently: \( \lambda \) is either a successor cardinal or a limit cardinal of uncountable cofinality).  
Then the isometry relation on (discrete) complete (ultra)metric spaces of density character \( \kappa \) is \emph{complete} for \( \kappa \)-analytic equivalence relations, i.e.\ every such equivalence relation is \( \kappa^+ \)-Borel reducible to it.
\end{theorem}

\begin{proof}
Set \( \alpha = \omega+ \omega+2 \). In~\cite[Theorem 2.10]{Hyttinen:2012fj}, it is shown that under our assumptions the relation of isomorphism between trees in \( \T{\alpha}{\kappa} \) is \( \leq^\kappa_B \)-above every 
\( \kappa \)-analytic equivalence relations on \( \pre{\kappa}{2} \). Arguing as in the first part of Proposition~\ref{prop:reductiontostandard}, one can see that every \( \kappa \)-analytic equivalence relation 
on an arbitrary \( \kappa \)-Borel 
spaces (which may be construed as a subspace of \( \pre{\kappa}{2} \)) is the restriction of a \( \kappa \)-analytic equivalence relation on a \( \kappa^+ \)-Borel subset of \( \pre{\kappa}{2} \), and thus it is also the 
restriction of a \( \kappa \)-analytic equivalence relation on the whole \( \pre{\kappa}{2} \). Thus the result of Hyttinen and Kulikov may be rephrased by saying that every \( \kappa \)-analytic equivalence relation (with 
an arbitrary Borel \( \kappa \)-space as domain) is \( \kappa^+ \)-Borel reducible to isomorphism on \( \T{\alpha}{\kappa} \). Therefore it is enough to show that isomorphism on \( \T{\alpha}{\kappa} \) is \( \kappa^+ \)-Borel 
reducible to isometry between discrete ultrametric spaces in \( \MS_\kappa \).

Let \( D \subseteq \RR_{> 0} \) be any bounded away from \( 0 \) countable set containing a strictly decreasing 
sequence of reals \( (r_\beta)_{\beta < \alpha } \) of length \( \omega+\omega+2 \), and let \( T \mapsto X_T \) be 
the map defined at the beginning of Section~\ref{sec:treesisometry}. Such a map can clearly be construed as a \( \kappa^+ \)-Borel functions between the Borel \( \kappa \)-spaces 
\( \T{\alpha}{\kappa} \) and \( \MS_\kappa \), and it reduces isomorphism 
to isometry by Theorem~\ref{thm:red}. Since all the spaces \( X_T \) are ultrametric and discrete (because \( D \) is bounded away from \( 0 \)), we are done.
\end{proof}

\begin{remark}
In the proof of Theorem~\ref{thm:mainkappa} it is crucial that we start from trees in \( \T{\alpha}{\kappa} \) with \emph{\(\alpha\) countable}: this is because in the construction from Section~\ref{sec:treesisometry} 
we need to consider strictly 
decreasing sequences of reals of length \( \alpha \), and by the separability of \( \RR \) any such sequence must have countable length. 
This is also the reason why one cannot generalize Gao-Kechris' proof of Theorem~\ref{thm:lowerboundsfordiscrete}\ref{thm:lowerboundsfordiscrete-2}
to directly show, in \( \mathsf{ZFC} \) alone, that \( \cong^\kappa_{\mathsf{TREES}} \) is \( \kappa^+ \)-Borel reducible to isometry on complete ultrametric spaces of density character \( \kappa \) 
(thus showing the optimality of the classification result in Theorem~\ref{thm:upperboundultrametrickappa}\ref{thm:upperboundultrametrickappa-2} under the additional assumption \( \kappa^{< \kappa} = \kappa \)).
\end{remark}

The assumption \( \mathsf{V = L} \) in Theorem~\ref{thm:mainkappa} cannot be completely removed, as there are models of \( \mathsf{ZF+DC} \) in which there are even clopen \( \kappa^+ \)-Borel equivalence relations 
on the generalized Baire space 
\( \pre{\kappa}{2} \) which are not reducible to e.g.\ the isometry relation on complete ultrametric spaces of density character \( \kappa \).

\begin{theorem}[\( \mathsf{ZF} \)] \label{thm:mainL(R)}
Assume \( \mathsf{AD + V=L(\RR)} \). Then for every uncountable \( \kappa \) there is a clopen equivalence relation \( E \) on \( \pre{\kappa}{2} \) such that there is no reduction of \( E \) to isometry between ultrametric and/or discrete spaces in \( \MS_\kappa \).  In particular, \( E \) is not \( \kappa^+ \)-Borel reducible to such isometry relations.
\end{theorem}

\begin{proof}
Fix any uncountable \( \kappa \) and let \( E_G \) be any orbit equivalence relation induced by a Polish group \( G \) acting in a turbulent way on \( \pre{\Nat}{2} \). By a remarkable result of Hjorth 
(see~\cite[Theorem 9.18]{Hjorth:2000zr}), under our assumptions
there is no function \( f \colon \pre{\Nat}{2} \to \Mod^\kappa_\L \) (where \( \L \) is the graph language) reducing \( E_G \) to \( \cong^\kappa_{\mathsf{GRAPHS}} \). 

Let \( E \) be the equivalence relation on \( \pre{\kappa}{2} \) defined by
\[ 
x \mathrel{E} y \iff x \restriction \omega \mathrel{E_G} y \restriction \omega.
 \] 
It is clear that \( E \) is clopen with 
respect to the bounded topology \( \tau_b \) on \( \pre{\kappa}{2} \). Assume towards a contradiction that there is a reduction of \( E \) to isometry on ultrametric and/or 
discrete spaces. Then by Theorem~\ref{thm:upperboundultrametrickappa} there would also be a reduction \( g \colon \pre{\kappa}{2} \to \Mod^\kappa_\L \) of \( E \) to \( \cong^\kappa_{\mathsf{GRAPHS}} \). 
Let \( f \colon \pre{\Nat}{2} \to \Mod^\kappa_\L \) be defined by setting \( f(z) = g(z \! \uparrow \! {}^ \kappa ) \) for every \( z \in \pre{\Nat}{2} \),
where \( z \! \uparrow \! {}^\kappa \in \pre{\kappa}{2} \) is defined by
\[ 
z \! \uparrow \! {}^\kappa (\alpha) = 
\begin{cases}
z(\alpha) & \text{if } \alpha < \omega \\
0 & \text{otherwise.}
\end{cases}
 \] 
Then for every \( z_0,z_1 \in \pre{\Nat}{2} \)
\[ 
{z_0 \mathrel{E_G} z_1} \iff { z_0 \! \uparrow \! {}^\kappa \mathrel{E}  z_1 \! \uparrow \! {}^\kappa} \iff {g( z_0 \! \uparrow \! {}^\kappa) \cong g( z_1 \! \uparrow \! {}^\kappa)},
 \] 
that is \( f \) would reduce \( E_G \) to \( \cong^\kappa_{\mathsf{GRAPHS}} \), contradicting Hjorth's result.
\end{proof}

Notice however that 
Theorem~\ref{thm:mainL(R)} leaves open the problem of determining the exact complexity of the isometry relation \( \cong^i \) on the whole \( \MS_\kappa \) when assuming \( \mathsf{ZF+AD+V=L(\RR)} \). 

\begin{remark}
The tension between Theorems~\ref{thm:mainkappa} and~\ref{thm:mainL(R)} seems to suggest that
in the uncountable context it could be more interesting to use ``definable'' functions 
(e.g.\ functions in some suitable well-behaved inner model like \( \mathrm{L}(\RR) \) or \( \mathrm{HOD} \)) rather than arbitrary \( \kappa^+ \)-Borel functions to classify uncountable/non-separable mathematical 
objects, even when working in the full \( \mathsf{ZFC} \). 
In fact, (the proof of) Theorem~\ref{thm:mainL(R)} shows that if we assume \( \mathsf{AD}^{\mathrm{L}(\RR)} \), then for every uncountable cardinal \( \kappa \) there 
is a clopen equivalence relation \( E \) on \( \pre{\kappa}{2} \) which cannot be reduced to isometry on ultrametric and/or discrete spaces in \( \MS_\kappa \) via \emph{any function in \( \mathrm{L}(\RR) \)}.
(Recall that \( \mathsf{AD}^{\mathrm{L}(\RR)} \) follows from both large cardinal assumptions, and strong forcing axioms, like \( \mathsf{PFA} \).)
On the other hand, it is not hard to show (heavily using the axiom of choice \( \mathsf{AC} \)) 
that if \( \kappa^{< \kappa} = \kappa \), then any \( E \) as above is \emph{\( \kappa^+ \)-Borel reducible} to \( \cong^\kappa_{\mathsf{GRAPHS}} \) and hence, 
by Theorem~\ref{thm:graphskappaanddiscrete}, to isometry between discrete complete metric spaces of density character \( \kappa \); 
however, any witness of the latter cannot belong to \( \mathrm{L}(\RR) \) (at least in presence 
of large cardinals or forcing axioms), and thus it arguably fails to be sufficiently ``concrete'' as an assignment of complete invariants.%
\footnote{This situation is absent in the countable/separable context, as any Borel function between two standard Borel (\( \omega \)-)spaces 
is definable using reals and (countable) ordinals as parameters, and thus it belongs to any inner model containing the same reals of the universe, 
such as \( \mathrm{L}(\RR) \).}
\end{remark}

Let us conclude this section by further mentioning 
that (anti-)classification results similar to Theorems~\ref{thm:graphskappaanddiscrete} and~\ref{thm:mainkappa} 
can be obtained for other natural mathematical objects, such as the non-separable Banach spaces (up to linear isometry); see the forthcoming~\cite{AM} for more on this.

\section{Further extensions of the theory of Borel reducibility} \label{sec:finaldiscussion}

The (classical) theory of Borel reducibility is usually concerned with \emph{analytic} equivalence relations on \emph{standard Borel} spaces. Besides its generalization to the uncountable/non-separable setup 
considered in  
Section~\ref{sec:nonseparable}, there are at least two directions in which the theory could be expanded in order to be able to handle more and more classification problems from other areas of mathematics.

The first extension would be that of considering also analytic equivalence relations with more complicated (i.e.\ non-standard) Borel domains.  To the best of our knowledge very little is known in this direction, but as we have seen in 
 Proposition~\ref{prop:discreteandlocallycompactarecoanalytic}
and Remark~\ref{rmk:nonBorel} this is often a necessary move. A basic observation is that 
when considering the classes of analytic equivalence relations considered in this paper, then
analytic equivalence relations with \emph{standard} Borel domains 
are at least \( \leq_B \)-cofinal among analytic equivalence relations with arbitrary Borel domains.

\begin{proposition} \label{prop:reductiontostandard} 
Let \( E \) be an analytic equivalence relation on a Borel space \( X \), and suppose that  \( E \leq_B H \) for some equivalence relation \( H \) with standard Borel domain.
Then there is a standard Borel space \( Y \supseteq X \) and an analytic equivalence relation \( F \) on \( Y \) such that \( E \) is the restriction of 
\( F \) to \( X \) and \( F \leq_B H \).
\end{proposition}

In particular, if \( E \) is a smooth (respectively, essentially countable, classifiable by countable structures, essentially orbit) equivalence relation on a Borel space, then it can be construed as 
the restriction of a smooth (respectively, essentially countable, classifiable by countable structures, essentially orbit) equivalence relation with standard Borel domain.

\begin{proof}
Let \( f \colon X \to Z \) be a Borel reduction of \( E \) to \( H \), where \( Z \) is the standard Borel domain of \( H \). By a theorem of Kuratowski (see~\cite[Theorem 12.2]{Kechris1995}), 
there is a Borel extension \( \hat{f} \colon Y \to Z \) of \( f \), where \( Y  \) is any standard Borel space containing \( X \). Setting for \( x,y \in Y \)
\[ 
x \mathrel{F} y \iff \hat{f}(x) \mathrel{H} \hat{f}(y),
 \] 
we get that \( F \cap (X \times X) = E \) and \( \hat{f} \colon F \leq_B H \), as desired.
\end{proof}

However, considering arbitrary Borel domains instead of just the standard ones strictly enlarges the scope of the theory of Borel reducibility, in the sense that it yields new \( \sim_B \)-equivalence classes. An example 
of this kind is the relation of isomorphism between well-founded trees on \( \Nat \) (or, equivalently, the isometry relation between discrete or locally compact closed subspaces of \( \pre{\Nat}{\Nat} \)).

\begin{proposition} \label{prop:isoWF}
The isomorphism relation on the collection \( \mathrm{WF} \) of well-founded trees on \(\Nat\) is not Borel bi-reducible with any equivalence relation with analytic (hence also standard Borel) domain.
\end{proposition}

\begin{proof}
Assume towards a contradiction that \( E \) is an equivalence relation with analytic domain \( X \) which is Borel bi-reducible with isomorphism on \( \mathrm{WF} \), and let
\( f \colon X \to \mathrm{WF} \) and \( g \colon \mathrm{WF} \to X \) be Borel witnesses of this fact. Since the range of \( f \) is an analytic subset of \( \mathrm{WF} \), by~\cite[Theorem 35.23]{Kechris1995} 
it is contained in \( \mathcal{T}_\alpha \) 
for some \( 1 \leq \alpha < \omega_1 \) (recall that \( \mathcal{T}_\alpha \) is the collection of well-founded trees on \( \Nat \) of rank \( < \alpha \), 
see the paragraph before Fact~\ref{fct:cofinal}). Therefore \( f \circ g \) would be a Borel reduction of isomorphism on \( \mathrm{WF} \) to the relation \( \cong_\alpha \) of isomorphism on \( \mathcal{T}_\alpha \), whence its restriction to \( \mathcal{T}_{\alpha+1} \) would witness \( {\cong_{\alpha+1}} \leq_B { \cong_\alpha} \), contradicting Fact~\ref{fct:cofinal}.
\end{proof}

Another direction in which the theory of Borel reducibility could be expanded is that of considering possibly more complicated equivalence relations, like the projective ones. With the remarkable exception of coanalytic equivalence relations~\cite{Silvercoanalytic,Sterncoanalytic,Hjorthcoanalytic,Gaocoanalytic}, 
the possibility of analyzing non-analytic equivalence relations using the 
preorder \( \leq_B \) (or suitable variants of it)  has received little attention in the literature, usually with a focus on the theoretical aspects (e.g.\ number of equivalence classes, see~\cite{Kechrisprojective,Harringtonprojective,Schlichtprojective}) rather than on the determination of the complexity of concrete examples. 
However, the quest for such a generalization is strongly motivated by the fact 
that several classification problems are represented by non-analytic equivalence relations (on suitable spaces of codes), including the following important examples.

\begin{itemizenew}
\item
Clemens~\cite{clemensborel} showed that 
the classification of Borel automorphisms of  \( \pre{\Nat}{2} \) up to conjugacy is a proper \( \mathbf{\Sigma}^1_2 \) equivalence relation, and thus it is certainly neither analytic nor coanalytic.
\item
As observed in~\cite{louros}, the equivalence relation of Borel bi-reducibility \( \sim_B \) between analytic equivalence relations (with standard Borel domains) is \( \mathbf{\Sigma}^1_3 \) in the codes. Adams and Kechris~\cite{adamskechris} indeed showed that the restriction of \( \sim_B \) to countable Borel equivalence relations 
is already a proper \( \mathbf{\Sigma}^1_2 \) equivalence relation.
\item
The homeomorphism problem for Polish metric spaces (that is, the equivalence relation on \( \MS_\Sep \) in which two spaces are equivalent if and only if they are homeomorphic) is \( \mathbf{\Sigma}^1_2 \)  but probably non-analytic, although its restriction to compact spaces is analytic 
(see e.g.\ \cite[Propositions 14.4.2 and 14.4.3]{gaobook}).
\end{itemizenew}

\end{document}